\documentclass[12pt]{amsart}
\usepackage[margin=1in]{geometry}
\usepackage{amsmath,amsfonts,amsthm,amssymb,bbm}
\usepackage{graphicx,color,dsfont}
\usepackage{enumitem}
\usepackage{fourier}

\newtheorem{theorem}{Theorem}
\newtheorem{lemma}{Lemma}
\newtheorem{proposition}{Proposition}

\newtheorem{definition}{Definition}
	
	\newtheorem{corollary}{Corollary}

\newtheorem{claim}{Claim}

 \theoremstyle{definition}
 
 \theoremstyle{remark}

 \numberwithin{equation}{section}

\newcommand{\vertiii}[1]{{\left\vert\kern-0.25ex\left\vert\kern-0.25ex\left\vert #1
    \right\vert\kern-0.25ex\right\vert\kern-0.25ex\right\vert}}

\newcommand{\f}[2]{\frac{#1}{#2}}

\newcommand{\cl}{{\mathcal L}}

\newcommand{\al}{\alpha}

\newcommand{\ga}{\gamma}

\newcommand{\de}{\delta}

\newcommand{\la}{\lambda}

\newcommand{\si}{\sigma}

\newcommand{\vp}{\varphi}

\newcommand{\om}{\omega}

\newcommand{\rone}{\mathbb R}

\newcommand{\dpr}[2]{\langle #1,#2 \rangle}

\newcommand{\eps}{\epsilon}

\newcommand{\cm}{\mathcal M}

\newcommand{\cd}{\mathcal D}

\newcommand{\ch}{\mathcal H}

\newcommand{\p}{\partial}

\newcommand{\beq}{\begin{equation}}
\newcommand{\eeq}{\end{equation}}
\newcommand{\beqna}{\begin{eqnarray*}}
\newcommand{\eeqna}{\end{eqnarray*}}
\newcommand{\beqn}{\begin{equation*}}
\newcommand{\eeqn}{\end{equation*}}
\newcommand{\bp}{\begin{proof}}
\newcommand{\ep}{\end{proof}}
\newcommand{\bprop}{\begin{proposition}}
\newcommand{\eprop}{\end{proposition}}
\newcommand{\bt}{\begin{theorem}}
\newcommand{\et}{\end{theorem}}
\newcommand{\bex}{\begin{Example}}
\newcommand{\eex}{\end{Example}}
\newcommand{\bc}{\begin{corollary}}
\newcommand{\ec}{\end{corollary}}
\newcommand{\bcl}{\begin{claim}}
\newcommand{\ecl}{\end{claim}}
\newcommand{\bl}{\begin{lemma}}
\newcommand{\el}{\end{lemma}}

\newcommand{\cj}{{\mathcal J}}

\begin{document}

\title
[Existence and stability for the travelling waves of the Benjamin equation]{ Existence and stability for the travelling waves of the Benjamin equation}

 	\author{Sevdzhan Hakkaev} 
 \address{Sevdzhan Hakkaev, ORCID 0000-0002-1531-8340, Department of Mathematics, Faculty of Science, Trakya University, \\ 22030 Edirne, Turkey and \\ Institute of Mathematics and Informatics, \\ Bulgarian Academy of Sciences, Acad. G. Bonchev Str. bl. 8, 1113 Sofia, Bulgaria}
 \email{s.hakkaev@shu.bg}
 
 \author[Milena Stanislavova]{\sc Milena Stanislavova}
 \address{Milena Stanislavova, ORCID 0000-0002-7144-3613, Department of Mathematics, University of Alabama - Birmingham,
 	1402 10th Avenue South
 	Birmingham AL 35294, USA
 }
 \email{mstanisl@uab.edu}
 \author[Atanas G. Stefanov]{\sc Atanas G. Stefanov}
 \address{Atanas G. Stefanov, ORCID 0000-0002-0934-506X\\  Department of Mathematics, University of Alabama - Birmingham,
 	1402 10th Avenue South
 	Birmingham AL 35294, USA.
 }
\email{stefanov@uab.edu}


\date{\today}

\begin{abstract}
   In the seminal work of Benjamin,\cite{Ben} in the late 70's,  he has derived the ubiquitous Benjamin model, which is a reduced model in the theory of  water waves. Notably,  it  contains two parameters in its dispersion part and  under some special circumstances, it  turns into  the celebrated KdV or the Benjamin-Ono equation, During the90's, there was renewed interest in it.   Benjamin, \cite{Ben1}, \cite{Ben2} studied the problem for existence of solitary waves, followed by works of  Bona-Chen, \cite{BC}, Albert-Bona-Restrepo, \cite{ABR}, Pava, \cite{Pava1}, who have showed the existence of travelling  waves, mostly by variational, but also bifurcation methods. Some results about the stability became available, but unfortunately, those were restricted to either small waves or Benjamin model, close to a distinguished (i.e. KdV or BO) limit.  Quite recently, in \cite{ADM},    Abdallah, Darwich and Molinet, proved existence,  orbital stability and  uniqueness results for  these waves, but only  for  large values of  $\f{c}{\ga^2}>>1$.

   In this article, we present an alternative constrained maximization procedure for the construction of these waves, for the full range of the parameters, which allows us to ascertain their spectral stability. Moreover, we extend this construction   to all $L^2$ subcritical cases (i.e. power nonlinearities $(|u|^{p-2}u)_x$, $2<p\leq 6$). Finally, we propose a different procedure, based on a specific form of the Sobolev embedding inequality, which works for all powers $2<p<\infty$, but produces some unstable waves, for large $p$. Some open questions and a conjecture regarding this last result  are proposed for further investigation. 
   	
\end{abstract}

\maketitle

  \section{Introduction}
  We consider the Benjamin equation 
  \begin{equation}
  	\label{4} 
  	u_t-u_{xxx}-\ga D u_x + 2 u u_x=0, \ \ (t,x)\in \rone_+\times \rone. 
  \end{equation}
   where $D=H\p_x$, where $H$ is the Hilbert transform. For equivalent Fourier symbol representation, see \eqref{a:11} below. 
   
   The model \eqref{4}  has been studied quite a bit since it first appeared, more than 50 years ago, in the work of Benjamin, \cite{Ben}. In this work, Benjamin used asymptotic scaled expansions,  and showed its derivation from the basic water wave models.  Benjamin came back to the study of the traveling waves some twenty years later in \cite{Ben1}, \cite{Ben2}, in which he showed the existence of these waves via degree type arguments, while the stability question remained largely unadressed. Travelling waves for the full Benjamin equation were also previously constructed, in the late 90's, by Bona and Chen, \cite{BC} and Albert-Bona-Restrepo, \cite{ABR}, with some orbital stability results (for waves of Benjamin, where the equation is close to the KdV limit) appearing in \cite{ABR}. Some more stability results were proved around the same time \cite{Pava1}, for waves with small $L^2$ norm. Very recently, in a work of Abdallah-Darwich-Molinet, \cite{ADM}, the authors employed a different constrained minimization scheme, which allowed them to show uniqueness and orbital stability for waves with sufficiently large, relative to the dispersion parameter $\ga$ wave speeds.

\subsection{An equivalent rescaled form of the Benjamin equation}
We perform a rescaling argument of the equation \eqref{4}.  To this end,   let  $u(t,x)=\phi(x+c t)$, with $\phi$ vanishing at both $\pm \infty$. Plugging this ansatz in \eqref{4}, and after taking into account the vanishing of $\phi$, we can integrate once (with constant of integration equal to zero). We obtain the profile equation
\begin{equation}
	\label{10} 
	-\phi''-\ga D\phi+c \phi +\phi^2=0.
\end{equation}
   A rescaling transformation like $\phi\to -\f{\ga}{2} \phi(\f{\ga}{2}\cdot)$ will transform the problem into an equivalent one  
    \begin{equation}
   	\label{11} 
   	-\phi''- 2 D\phi+4c \ga^{-2} \phi - \phi^2=0.
   \end{equation}
   A well-known principle in physics suggests that well-localized solutions of \eqref{11} do not support eigenvalues embedded inside the essential spectra. As zero mode is always present, one should expect solutions to  exists only if the dispersion relation, 
   $\xi^2 - 2 |\xi| +4c\ga^{-2}>0$ for all values of $\xi$. This  implies that $4c\ga^{-2}>1$ or $c>\f{\ga^2}{4}$. This is a well-known existence condition, imposed by   all previous works, \cite{BC}, \cite{Pava1},  \cite{ABR}, \cite{ADM}. 
   
    Clearly, the rescaled equation \eqref{11} depends on one parameter only, which in view of the requirement $4c\ga^{-2}>1$, can be written via  the new parameter $\om:=4c\ga^{-2}-1>0$. Henceforth, we consider an equivalent version of the problem  \eqref{11}
   \begin{equation}
   	\label{12} 
   	-\phi''- 2 D\phi+(\om+1)  \phi - \phi^2=0.
   \end{equation}
   Note that due to the fact  $D^2=H^2\p_x^2=-\p_x^2$, we can equivalently rewrite in the more convenient form 
   \begin{equation}
   	\label{14} 
   	(D-1)^2 \phi+\om \phi- \phi^2=0.
   \end{equation}
An associated change of variables in the Benjamin equation \eqref{4}, namely 
$$
u(t,x)\to -\f{\ga}{2} u\left(\left(\f{\ga}{2}\right)^3 t, \f{\ga}{2}(x+c t)\right),
$$
 leads to the following convenient version of the Benjamin equation, in its  traveling wave ansatz
\begin{equation}
	\label{b:10} 
	u_t +\p_x ((D-1)^2 u+\om u  - u^2)=0.
\end{equation}
Clearly now the solutions $\phi$ of \eqref{14} are nothing but steady  solutions of \eqref{b:10}, and the study of the stability of the traveling wave solutions of \eqref{4} reduces to the stability of the corresponding steady states for \eqref{b:10}. 
    To this end,  we consider perturbations of such steady states and their respective stability properties. Specifically, let $u(t,x)=\phi(x)+ v(t,x)=\phi(x)+ e^{\la t} v$ and plug  this in \eqref{b:10}. After ignoring $O(v^2)$ terms, we arrive at the linearized/eigenvalue  problem 
  $$
    	\p_x ((D-1)^2 +\om - 2 \phi) v=-\la v
   $$
   Upon introducing the self-adjoint linearized operator\footnote{Here, we need to know some {\it a priori} decay  properties of the potential $\phi$, in order to have a well-defined self-adjoint operator. As we shall see later, $\phi$ is a smooth function, with somewhat limited decay, namely $|\phi(x)|\leq C (1+x^2)^{-1}$. This is however  enough for our purposes.}
   $$
   \cl_+:= (D-1)^2 +\om - 2 \phi, D(\cl_+)=H^2(\rone), 
   $$
  we can view the eigenvalue problem in the form $\p_x \cl_+ v =-\la v$. Due to its Hamiltonian structure (and the associated spectral symmetry with respect to real and imaginary axes), we arrive at the equivalent spectral problem 
  \begin{equation}
  	\label{b:15} 
  	\p_x \cl_+ v=\la v.
  \end{equation}
   This type of Hamiltonian eigenvalue problem (where $\p_x$ is a particular skew symmetric operator, while $\cl_+$ is self-adjoint) is sometimes referred to as KdV-type eigenvalue problem.   
   Before we continue with our presentation,  let us take a derivative in the spatial variable in \eqref{14}, and record the result in the form
   \begin{equation}
   	\label{17} 
   \cl_+[\phi']=0.
   \end{equation}
In other words, $span[\phi']\subset Ker[\cl_+]$. It is a standard expectation\footnote{but unfortunately hard to rigorously verify in specific examples, especially when  non-local operators are involved, the case here} that this is the only element of $Ker[\cl_+]$, see Definition \ref{defi:10} below.   
Next, we recall various standard notions. 
   \begin{definition}
   	\label{defi:10} 
   	The traveling wave solution $\phi$ is \underline{spectrally stable}, if 
   	$\si(\p_x \cl_+)\subset i\rone$. 
   	Equivalently, $\phi$ is  spectrally unstable, if the eigenvalue problem \eqref{b:15} possesses  a non-trivial solution, with $\Re\la>0, v\in H^3(\rone), v\neq 0$.  
   	
   	The wave $\phi$ is called \underline{non-degenerate}, if in fact $Ker[\cl_+]=span[\phi']$. We say that $\phi$ is weakly non-degenerate, if $\phi\perp Ker[\cl_+]$. 
   \end{definition}
\subsection{Extensions to generalized Benjamin models}
Following similar extensions considered for other fluid equations, such as KdV and Benjamin-Ono, and in the context of \eqref{b:10}, we consider the focusing generalized Benjamin equation
\begin{equation}
	\label{b:30} 
	u_t +\p_x ((D-1)^2 u+\om u  - |u|^{p-2} u)=0, p>2.
\end{equation}
 Let us mention that the related version, with non-linearity $u^{p-1}$ and integer $p$ is also a possibility, but since the analysis is, for all practical purposes similar, we concentrate on \eqref{b:30}. Localized steady states of such models satisfy an elliptic problem similar to \eqref{14}. Note that we expect $\phi$ to be a sign changing solution\footnote{numerics in \cite{ABR} confirm that,  in the classical case of quadratic nonlinearity},   so we need to keep the absolute value in the non-linearity.  Specifically, we obtain 
 \begin{equation}
 	\label{b:50} 
 	(D-1)^2\phi +\om \phi - |\phi|^{p-2} \phi=0,
 \end{equation}
and once existence is established, one can study the associated stability problem. This takes the form 
\begin{eqnarray}
	\label{b:60} 
	\p_x\cl_+ v=\la v,
\end{eqnarray}
where $ \cl_+=(D-1)^2+\om - (p-1) |\phi|^{p-2}$. 
We are now ready to state the main results. 
\subsection{Main results}
We present first the result about the classical Benjamin model, in its equivalent form \eqref{b:10}. 
\subsubsection{Existence and stability of the Benjamin travelling waves}
\begin{theorem}
	\label{theo:10} 
	There exists a one parameter family of functions $\{\phi_\al\}_{\al>0}$, with the following properties 
	\begin{itemize}
		\item $\phi_\al\in H^\infty(\rone)$ (see also \cite{BL}, \cite{ADM}) 
		solves \eqref{14} with 
		\begin{equation}
			\label{c:10} 
			\om=\om_\al=\al+\f{\|(D-1)\phi_\al\|^2+\al \|\phi_\al\|^2}{2\|\phi_\al\|^2}. 
		\end{equation}
	Moreover, there is the asymptotic\footnote{This also appears in  \cite{BL}, \cite{ADM}, so we will not dwell too much on its proof. The sharp decay rate of $|\phi$ is inherited from  the decay of the Green's function $G: \hat{G}=\f{1}{\xi^2-2|\xi|+\om+1}$, which turns out to be $|G(x)|\leq C(1+|x|^2)^{-1}$.}
	$$
	\lim_{x\to \pm \infty} |x|^2 |\phi(x)|=K\neq 0,
	$$
	
		\item Each $\phi_\al, \al>0$ is spectrally stable steady solution of \eqref{b:10} with $\om=\om_\al$, in the sense of Definition \eqref{defi:10}. 
		\item the mapping $\al\to \|\phi_\al\|$ is continuous and strictly increasing, with 
		$$
		\lim_{\al\to 0+} \|\phi_\al\|=0, \ \ \lim_{\al\to +\infty} \|\phi_\al\|=+\infty.
		$$
	\end{itemize}
\end{theorem}
{\bf Remark:} 
\begin{itemize}
	\item The waves $\phi_\al$ are constructed to saturate a  Gagliardo-Nirenberg type inequality, see \eqref{20} below. 
	\item As an increasing function, the mapping $\al\to \|\phi_\al\|^2$ is differentiable almost everywhere and $\p_\al \|\phi_\al\|^2\geq 0$. 
	\item The uniqueness of these waves was established in \cite{ADM} for the cases where $\f{c}{\ga^2}>>1$. According to our rescaling, this is   equivalent to the condition $\om>>1$ (or equivalently the large $L^2$ norm  waves). That is, for large enough $\om$,  the waves constructed herein coincide with the waves exhibited in all the other earlier works. As shown in \cite{ADM}, these are also orbitally stable, which matches our results above. We note that we actually show more, namely the stability of these waves is true for all $\om>0$, not just the large ones. 
	\item The condition $\om>0$ is equivalent to the condition $c >\f{\ga^2}{4}$ in the previous works, \cite{BC}, \cite{Pava1}, \cite{ADM}.  Most likely, waves do not exist in the regime $c\leq \f{\ga^2}{4}$. One way to see that is to rule out the point zero\footnote{Note that since  $\cl_+[\phi']=0$, $0\in\si_{p.p.}(\cl_+)$.} as embedded or edge spectrum for the linearized operator $\cl_+$. This is usually done by establishing a limiting absorption principle for operators of this dispersive type, but  this is an open problem at the moment. 
\end{itemize}
As is usual in calculus of variations (and since the waves $\phi_\al$ are constructed via a variational approach), it is hard to control the smoothness of the relation $\al\to \om_\al$ or more generally  $\al\to \vp_\al$. Nevertheless, {\it and under the assumption that the mapping $\al\to \phi_\al$ is Frech\'et differentiable}, 
we can still say conclude that  the map $\al\to \om_\al$ is monotonically increasing. We have the following proposition. 
\begin{proposition}
	\label{prop:100} 
	Assume that $\al\to \phi_\al$, as a mapping $(0, +\infty) \to H^2(\rone)$  is Frech\'et differentiable. In particular (as it is increasing function according to Theorem \ref{theo:10}), 
	$$
	\p_\al \|\phi_\al\|^2\geq 0.
	$$
	Then, 
	the mapping $\al\to \om_\al$ is differentiable. Moreover, either $\p_\al\|\phi_\al\|^2=0$ or otherwise if $\p_\al \|\phi_\al\|^2>0$, then $\om'(\al)>0$.

\end{proposition}
\subsubsection{Existence and stability of the travelling waves for  the generalized Benjamin equation}
The main result here looks very similar to the classical one, Theorem \ref{theo:10}, in the $L^2$ subcritical regime, $p\in (2, 6]$. 
\begin{theorem}
	\label{theo:20} 
Let $p\in (2, 6]$. Then,  there exists a one parameter family of functions $\{\phi_{\al,p}\}_{\al>0}$, with the following properties 
	\begin{itemize}
		\item $\phi_{\al,p} \in H^3(\rone): |\phi(x)|\leq \f{C}{1+|x|^2}$,  solves \eqref{b:50} with 
		\begin{equation}
			\label{c:101} 
			\om=\om_{\al,p}= \f{p \al}{2} +\f{(p-2) }{2} \f{\|(D-1)\phi_{\al,p}\|^2}{\|\phi_{\al,p}\|^2}
		\end{equation}
	
		\item Each $\phi_\al, \al>0$ is spectrally stable steady solution of \eqref{b:30} with $\om=\om_{\al,p}$. 
		
		\item the mapping $\al\to \|\phi_\al\|$ is continuous and strictly increasing, and 
		$$
		\lim_{\al\to 0+} \|\phi_{\al,p}\|=0, \ \ \lim_{\al\to +\infty} \|\phi_{\al,p}\|=+\infty.
		$$
	\end{itemize}
\end{theorem}
{\bf Remarks:} 
\begin{itemize}
	\item Identical results may be formulated for models with the  nonlinearities   $(u^{p-1})_x$, $p=4,5,6$, instead of $(|u|^{p-2} u)_x$. Naturally, the case $p=3$ being the classical one, has already been addressed in Theorem \ref{theo:10}.  
	\item The waves $\phi_{\al,p}$ are constructed to saturate an interpolation Gagliardo-Nirenberg type inequality, see \eqref{30} below. 
	\item  $\phi \in H^1(\rone)$ is deduced directly from the variational construction. This is immediately bootstrapped to $\phi \in H^2(\rone)$, from the elliptic equation. Taking further derivative, and the fact that $\p_x(|\phi|^{p-2}\phi)=(p-1) |\phi|^{p-2} \phi' \in L^2(\rone), p>2$ provides the further bootstrap $\phi_{\al,p}\in H^3(\rone)$. The space $H^3(\rone)$ provides satisfactory, if not optimal, level of smoothness, as it ensures that  $\phi'\in H^2(\rone)=D(\cl_+)$. 
	\item Related to the previous point,  the waves $\phi_{\al,p}$ may not be as regular (i.e. belong to a really high $H^s(\rone)$),  for non-integer values of $p$. This is   due to the limited smoothness of the mapping $z\to |z|^{p-2} z$. For even values of $p$ or the case of non-linearities $(u^{p-1})_x$, we do not encounter this, and we do have $\phi\in H^\infty(\rone)$,  due to the analyticity of $z\to z^{p-1}$. 
	
\end{itemize}
Our final result applies to waves with arbitrary power $p$. 
\begin{theorem}
	\label{theo:30}
	Let $\om>0$ and $2<p<\infty$. Then, there exist waves, $\phi_{\om,p}\in H^3(\rone)$, which solve \eqref{b:50}. They also obey  the asymptotics $|\phi(x)|\leq K(1+|x|^2)^{-1}$. 
	
Assuming that the waves $\phi$ are  non-degenerate ( i.e. $Ker[\cl_+]=span[\phi']$), 	there exists values of $p\in (2, +\infty)$, 
$ \om>0$, so that {\bf $\phi$ is spectrally unstable} as a steady solution of \eqref{b:30}. In particular, the condition 
\begin{equation}
	\label{850} 
	\f{p}{4}+\f{4}{p} > \f{1}{\om} + \f{5}{2}, 
\end{equation}
ensures the spectral  instability of the waves. 
\end{theorem}
{\bf Remarks:} 
\begin{itemize}
	\item The waves $\phi_{\om,p}$  saturate a Sobolev type inequality, see \eqref{510} below. In fact, our variational approach is similar\footnote{in spirit, if not in the precise technical details} to the one adopted in \cite{ADM}. 
	\item Note that the condition \eqref{850} does not yield any solutions for $p\in (2,6]$. This is in line with our previous considerations, as in this range we expect\footnote{according to Theorem \ref{theo:20} and the conjectured uniqueness of the solutions to \eqref{b:50}} 
	the waves to be spectrally stable. 
	\item The condition \eqref{850} is quite likely,  far from  a sharp criteria for the instability of the waves $\phi_{\om,p}$. Nevertheless, it guarantees that, say for $p>8$, and all large enough values of $\om$, the  wave $\phi_{\om,p}$ is unstable. It would be interesting to  prove that some of the waves for $p\in (6,8)$ are also   unstable.  This is unattainable with \eqref{850} alone, as $\f{p}{4}+\f{4}{p}<\f{5}{2}, p\in (6,8)$. 
	\item Related to the previous point, we conjecture that for each $p>6$, there exists a critical value $\om^*(p)$, so that the waves $\phi_{\om,p}$ are unstable for $\om>\om^*(p)$ and stable for $\om\in (0, \om^*(p))$.  The  methods in this  paper are blatantly insuffficient for such conclusions, but it may nevertheless provide an useful point of reference for future investigations. 
\end{itemize}

The plan for the paper is as follows. In Section \ref{sec:2}, we develop some preliminaries, including the interpolation inequalities needed in the sequel, a convenient version of the  compensated compactness lemma, as well as the Pohozaev's identities. In Section \ref{sec:KdV}, we revisit some stability vs. instability criteria for KdV-like eigenvalue problems. This part contains some new results on the said problems, which may be of independent interest, see specifically Corollary \ref{cor:10} and Proposition \ref{prop:43}. In Section \ref{sec:3}, we present the existence of the waves for  the classical Benjamin equation, as well as some spectral properties, which allow to deduce the stability of all such waves, see Proposition \ref{prop:40}. In Section \ref{sec:4.1}, we take on the question for the existence and stability of the waves for the Benjamin equation, with general power non-linearity. Proposition \ref{prop:80} provides the existence results for waves with $L^2$ subcritical powers (i.e. $2<p\leq 6$), while Proposition \ref{prop:65} shows that all such waves are spectrally stable.  Finally, in Section \ref{sec:4.2}, we construct the waves for all powers $2<p<\infty$ via a maximization of an appropriate Sobolev inequality. As it happens, such waves are not all stable\footnote{and in fact, we do not have rigorous stability results in this section, although conditional upon the uniqueness, we should have that these waves are stable in the $L^2$ subcritical range $2<p\leq 6$}. In fact, we show that under some relations of the parameters $p, \om$, these waves are unstable, see Lemma \ref{le:89}. 

\noindent {\bf Acknowledgements:} Milena Stanislavova is partially supported by the National Science Foundation,  under award \# 2210867. Stefanov acknowledges support from  the National Science Foundation, under award \# 2204788.\\
{\bf Data Availability statement:}  The authors declare that the data supporting the findings of this study are available within the paper. No additional data sets have been generated. \\
{\bf Competing interests:} The authors declare that they have no conflict of interest. 
   \section{Preliminaries} 
   \label{sec:2}
   We start with the function spaces and the Fourier transform. 
   \subsection{Function spaces and Fourier transform}
   We shall make use of the standard Lebesgue spaces $L^p$ with the usual  defintion. We denote the Hilbertian norm without any subscript $\|f\|:=\|f\|_{L^2}$. 
   We define the Fourier transform and its inverse via the formulas 
   $$
   \hat{f}(\xi)=\f{1}{\sqrt{2\pi}} \int_{-\infty}^{+\infty} f(x) e^{- i x\xi} dx;  \ \ \ f(x) = \f{1}{\sqrt{2\pi}}\int_{-\infty}^{+\infty} \hat{f}(\xi) e^{ i x\xi} d\xi
   $$
   Clearly, $\widehat{f'}(\xi)= i \xi \hat{f}(\xi), \widehat{f''}(\xi)=- |\xi|^2 \hat{f}(\xi)$. 
   
   The Hilbert transform is defined via the convolution with the distribution $p.v. \f{1}{x}$, i.e. $H f=p.v. \f{1}{x}*f$ or equivalently 
   $$
   H f(x)=\f{1}{\pi}  p.v. \int_{-\infty}^{+\infty} \f{f(x-y)}{y} dy.
   $$
   It can be computed that $\pi^{-1} \widehat{p.v. \f{1}{x}}(\xi)=-i sgn(\xi)$, whence $\widehat{H f}(\xi)=-i sgn(\xi) \hat{f}(\xi)$.  Note that 
   $$
   \widehat{H^2 f}(\xi)=(-i sgn(\xi))^2 \hat{f}(\xi)=-\hat{f}(\xi),
   $$
   or simply $H^2=-Id$. 
  Clearly then, the Zygmund operator $D=H \p_x=\sqrt{-\p_{xx}}$ introduced earlier may be written in the multiplier form 
   \begin{equation}
   	\label{a:11} 
   	 \widehat{D f}(\xi)=|\xi| \hat{f}(\xi).
   \end{equation}
  Note that by Plancherel's 
  $$
  \|D f\|= \left(\int |\xi|^2 |\hat{f}(\xi)|^2\right)^{\f{1}{2}}=\|f'\|.
  $$
   and $D^2=-\p_x^2$. Note that Sobolev spaces $W^{s,p}, 0<s<\infty, 1<p<\infty$, may be defined via the norms 
  $$
  \|f\|_{W^{s,p}}=\|f\|_{L^p}+\|D^s f\|_{L^p}.
  $$
   Note that we use the notation $H^s=W^{s,2}$. 
   
   Another classical estimate that will be useful in the sequel are the Kato-Ponce commutator estimates 
   \begin{equation}
   	\label{43} 
   		\|[D, g] f\|_{L^q}\leq C_q \|g'\|_{L^\infty}\|f\|_{L^q}, 1<q<\infty. 
   \end{equation}

   \subsection{Interpolation inequality}
  We start with a simple interpolation inequality. 
  \begin{lemma}
  	\label{le:10} 
  	Let $p=3,4,5,6$. Then, for every $\al>0$, there exists $C_{\al, p}$, so that 
  	\begin{eqnarray}
  		\label{20} 
  		\int_{-\infty}^\infty u^p(x) dx\leq C_{\al, p} \|u\|^{p-2} \left(\|(D-1) u\|^2+ \al \|u\|^2\right). 
  	\end{eqnarray}
  Moreover, for $2<p\leq 6$ and any  $\al>0$, there exists $D_{\al, p}$, so that 
  \begin{eqnarray}
  	\label{30} 
  	\int_{-\infty}^\infty |u(x)|^p dx\leq D_{\al, p} \|u\|^{p-2} \left(\|(D-1) u\|^2+ \al \|u\|^2\right). 
  \end{eqnarray}
  \end{lemma}
\begin{proof}
It suffices to establish \eqref{30}. 	Note that \eqref{30} is equivalent to 
$$
\|u\|_{L^p}\lesssim \|u\|^{1-\f{2}{p}} \|u\|_{H^1}^{\f{2}{p}} 
$$
since $\|(D-1) u\|^2+ \al \|u\|^2 \sim \|u\|_{H^1}^2$. 
By Sobolev inequality, followed by the Gagliardo-Nirenberg's inequality 
	\begin{equation}
		\label{25} 
			\|u\|_{L^p(\rone)}\lesssim  \|u\|_{\dot{H}^{\f{1}{2}-\f{1}{p}}(\rone)}\lesssim  \|u\|^{\f{1}{2}+\f{1}{p}}  \|u\|_{H^1}^{\f{1}{2}-\f{1}{p}} 
	\end{equation}
	 whence \eqref{30}  follows from \eqref{25} by  observing that $\f{1}{2}-\f{1}{p}\leq \f{2}{p}$ for $2<p\leq 6$, whence 
	 $$
	 	\|u\|_{L^p(\rone)}\lesssim \|u\|_{H^1}^{\f{2}{p}}  \|u\|^{1-\f{2}{p}},
	 $$
	 as required. 
\end{proof}

\subsection{Compensated compactness}
We state the relevant result from \cite{Lions}, adapted for the case of $\rone$. 
\begin{theorem}
	\label{theo:102} 
	Let $\rho_n: \rone\to \rone$, $\rho_n\geq 0, \int_{-\infty}^\infty \rho_n(x) dx = \la$. Then, there exists a subsequence $\rho_{n_k}$, so that one of the following is true 
	\begin{itemize}
		\item (Compactness/tightness) there exists a sequence $y_k\in \rone$, so that  for every $\epsilon>0$, there exists $R>0$, and $k_0$, so that for all $k\geq k_0$, 
		$$
		\int_{y_k-R}^{y_k+R} \rho_{n_k}(x) dx \geq \la-\epsilon.
		$$
		\item (Vanishing) For all $R<\infty$, 
		$$
		\lim_{k\to \infty} \sup_{y\in\rone} 	\int_{y-R}^{y+R} \rho_{n_k}(x) dx=0.
		$$
		\item (Dichotomy) There exists $\mu\in (0, \la)$, so that for all $\epsilon>0$, there exists $y_k\in\rone$ and $\rho_{k, +}, \rho_{k,-}\in L^1_+(\rone)$, so that for all $k$ large enough, $\rho_{n_k}(x+y_k)= \rho_{k,+}(x)+\rho_{k,-}(x)+e_k$ and 
		\begin{eqnarray*}
			& & supp (\rho_{k,-})\subset (-\infty, -R_k), supp(\rho_{k,+}) \subset (R_k, +\infty), \ \ \lim_k R_k=\infty\\
			& & |\int_{R_k}^{+\infty} \rho_{k,+}(x) dx - \mu|<\epsilon, 	|\int_{-\infty}^{-R_k} \rho_{k,-}(x) dx - (\la-\mu)|<\epsilon, \int |e_k(x)| dx<\epsilon.
		\end{eqnarray*}
	\end{itemize}
\end{theorem}

\subsection{Pohozaev's identities}
In this section, we develop the Pohozaev idenitities for the solutions of the elliptic problem \eqref{b:50}. To this end, we shall need the following useful commutation relation
\begin{equation}
	\label{b:60} 
	[D^s, x\p_x]=s D^s. 
\end{equation}
As we shall need to manipulate integrals  involving weights and $\phi', \phi''$, we  need a basic decay result for the derivatives $\phi', \phi''$. Such a statement is available, for example in \cite{ADM}, and it states 
$$
|\phi'(x)|+|\phi''(x)|\leq \f{C}{1+|x|^3}.
$$
In particular $x\phi'\in L^2(\rone)$. 
\begin{proposition}
	\label{prop:50} 
	
	Let $p>2$, $\om>0$ and $\phi\in H^3(\rone)$ is a solution of \eqref{b:50}. Then, we have the  Pohozaev's identities 
	\begin{eqnarray}
		\label{b:70} 
		& & 		\|\phi'\|^2- 2 \|D^{\f{1}{2}} \phi\|^2+(\om+1) \|\phi\|^2-\|\phi\|_p^p=0 \\
		\label{b:80} 
		& & \|\phi'\|^2 -(\om+1) \|\phi\|^2+\f{2}{p} \|\phi\|_p^p=0.
	\end{eqnarray}
\end{proposition}
\begin{proof}
	For the proof of \eqref{b:70}, we only need to take dot product of \eqref{b:50}, with $\phi$. 
	For \eqref{b:80}, we take dot product of \eqref{b:50} with 
	$x\phi'$. We have, after using $D^*=D$, 
	\begin{equation}
		\label{b:90} 
		\dpr{\phi}{D^2(x\p_x \phi)}-2 \dpr{\phi}{D(x\p_x \phi)}+(\om+1)\int x\phi'\phi dx- \int x |\phi|^{p-2} \phi\phi' dx=0.
	\end{equation}
	We have,   
	\begin{eqnarray*}
		\dpr{\phi}{D^2(x\p_x \phi)} &=& -\dpr{\phi''}{x\phi'}=\f{1}{2} \|\phi'\|^2.
	\end{eqnarray*}
	Next, by using \eqref{b:60} with $s=1$, we obtain 
	\begin{eqnarray*}
		\dpr{\phi}{D(x\p_x \phi)} &=&   \dpr{\phi}{x\p_x D\phi} + \dpr{\phi}{D \phi} = - \dpr{D(\phi+x\phi')}{\phi} +  \dpr{\phi}{D \phi} =\\
		&=& - \dpr{D(x\phi')}{\phi}
	\end{eqnarray*}
	We obtain 
	$
	\dpr{\phi}{D(x \phi')} =0.
	$
	
	Finally, elementary itegration by parts yields 
	$$
	\int x\phi'\phi dx = -\f{1}{2} \|\phi\|^2; \ \ \int x |\phi|^{p-2} \phi\phi' = -\f{1}{p} \|\phi\|_p^p.
	$$
	Putting it all together, we obtain \eqref{b:80}. 	 
\end{proof}


\section{Spectral stability for solitary waves of KdV-type problems}
\label{sec:KdV}
In this section, we revisit some well-known criteria for KdV-like problems, like the instability index theory, as presented in \cite{LZ}, see also earlier works \cite{KKS1}, \cite{KKS2}, \cite{Pel}. We  complement these results by some new ones, which are to the best of our knowledge  new (at least in the generality stated herein). These will be useful in the sequel. We start with the general theory, roughly following \cite{LZ}. 

\subsection{Introduction to instability index count}
Introduce first the notion of Morse index for a self-adjoint operator $S$ as the dimension of the negative subspace of $S$, to be denoted henceforth as $n(S)$. 

For an eigenvalue problem of the form 
\begin{equation}
	\label{e:10} 
	\cj\cl u=\la u,
\end{equation}
with $\cj^*=-\cj, \cl^*=\cl$, we introduce the generalized kernel 
$$
gKer(\cj\cl)=span\{u: (\cj\cl)^l u=0, l=1, 2, \ldots \}
$$
In the case of finite dimensionality of such linear subspace, i.e. $dim(gKer(\cj\cl))<\infty$, we introduce a basis $\{\eta_j\}_{j=1}^N$, and a symmetric matrix $\cd$, with entries 
$$
\cd_{i j}=\cd_{ji}=\dpr{\cl \eta_i}{\eta_j}
$$
On the other hand, introduce three different type of point spectra arising in the eigenvalue problem \eqref{e:10}, namely unstable real spectrum, denoted $k_r:=\{\la>0: \la\in \si_{p.p.}(\cj\cl)\}$, the unstable complex spetra, $k_c=\{\la: \la\in \si_{p.p.}(\cj\cl), \Re\la>0, \Im\la>0\}$ and finally, the marginally stable spectrum with negative Krein signature, defined by $k_{i}^-=\{i \mu, \mu>0: \cj\cl f=i \mu f, \dpr{\cl f}{f}<0 \}$. The instability index formula then reads 
\begin{equation}
	\label{e:20} 
	k_r+2 k_c+2k_i^-=n(\cl)-n(\cd).
\end{equation}
Note that the right-hand side of \eqref{e:20} overcounts the number of instabilities for the eigenvalue problem \eqref{e:10}. In the particular case $n(\cl)=1$, which will be of the main interest herein, the formula \eqref{e:20} yields a rather precise information about the stability for \eqref{e:10}, namely that the stability occurs exactly when $n(\cd)=1$. 
\subsection{Stability criteria for KdV-like eigenvalue problems}
\label{sec:2.4}
   Suppose that we consider an eigenvalue problem in the form 
   \begin{equation}
   	\label{200} 
   	\p_x \cl v = \la v,
   \end{equation}
where $\cl$ is a self-adjoint operator,  with $D(\cl)\subset L^2(\rone)$,  so that $\cl$ maps real valued into real valued functions. We have the following Proposition. 
\begin{proposition}
	\label{prop:30} 
	Assume that the eigenvalue problem \eqref{200} is spectrally 
	unstable and  $\cl[\psi']=0$ for some smooth {\it real-valued} function $\psi$. That is, there exists a $\la: \Re\la>0$ and $v\in D(\p_x \cl)$, so that \eqref{200} holds true. Then,
	$$
	n(\cl|_{\{\psi\}^\perp})\geq 1.
	$$
	Equivalently 
	$$
	\inf\{\dpr{\cl h}{h}: h\in D(\cl), h\perp \psi, \|h\|=1\}<0.
	$$
\end{proposition}
\begin{proof}
	Suppose that instability occurs. Since $\p_x \cl v = \la v$, take a dot product with $\psi$ 
	$$
	\la \dpr{v}{\psi}=\dpr{\p_x \cl v}{\psi}=-\dpr{ \cl v}{\psi'}=-\dpr{ v}{\cl   \psi'}=0.
	$$
	As $\la\neq 0$, it follows that $v\perp \psi$.  Then, $v=v_1+i v_2, v_1\perp \psi, v_2\perp \psi$,  $\la=\la_1+i \la_2, \la_1>0$. We have, in real and imaginary entries
	\begin{equation}
		\label{198} 
			\left\{
		\begin{array}{l}
			\p_x \cl  v_1=\la_1 v_1 - \la_2 v_2 \\
			\p_x \cl v_2 = \la_2 v_1 + \la_1 v_2 
		\end{array}
		\right.
	\end{equation}
	Take dot products with the real-valued functions $\cl  v_1, \cl v_2$ respectively and add.  The result is 
	$$
	\la_1\left(\dpr{\cl v_1}{v_1}+ \dpr{\cl  v_2}{v_2}\right)=0.
	$$
	and so 
	\begin{equation}
		\label{210} 
			\dpr{\cl  v_1}{v_1}+ \dpr{\cl  v_2}{v_2} =0.
	\end{equation}
Clearly, as one of the terms above is non-positive, 	$\inf\{\dpr{\cl h}{h}: h \perp \psi, \|h\|=1\}\leq 0$.  We claim that in fact,
	$$
	\inf\{\dpr{\cl h}{h}:  h\perp \psi, \|h\|=1\}<0,
	$$
	which is equivalent to  the statement of the Proposition. 
	Indeed, if one of the expressions in \eqref{210} is strictly negative, we are done, as $v_1\perp \psi,  v_2\perp \psi$. Otherwise, $\dpr{\cl  v_1}{v_1}= \dpr{\cl  v_2}{v_2} =0$ and 
	$$
 	\inf\{\dpr{\cl h}{h}:  h\perp \psi, \|h\|=1\}=0. 
	$$
	That is $P_{\{\psi\}^\perp}\cl P_{\{\psi\}^\perp}\geq 0$
	  But then, as $v_1\perp \psi,  v_2\perp \psi$, we conclude that $\cl  v_1=c_1 \psi; \cl  v_2=c_2\psi$. But in such a case, from \eqref{198}, 
		\begin{equation}
			\label{204} 
			\left\{
			\begin{array}{l}
				c_1 \psi'=\la_1 v_1 - \la_2 v_2 \\
				c_2\psi' = \la_2 v_1 + \la_1 v_2 
			\end{array}
			\right.
		\end{equation}
	 Multiplying the first equation by $\la_1$, the second one by $\la_2$ and adding results in 
	 $$
	 (\la_1^2+\la_2^2) v_1 = (\la_1 c_1+\la_2 c_2)\psi',
	 $$
	 whence $v_1=C_1 \psi'$ and hence $\cl v_1=0$. Similarly,  multiplying the first equation of \eqref{204} by $-\la_2$, the second one by $\la_1$ and adding leads to $v_2=C_2 \psi'$, and again $\cl v_2=0$. From \eqref{198}, we now have 
	 $$
	 \left\{
	 \begin{array}{l}
	 	0=\la_1 v_1 - \la_2 v_2 \\
	 	0= \la_2 v_1 + \la_1 v_2 
	 \end{array}
	 \right.
	 $$	 
which by resolving the same way as above (note $\la_1>0$), yields $v_1=v_2=0$, which is contradictory as well.  
	\end{proof} 
An immediate and useful corollary of Proposition \ref{prop:30} is then the following. 
\begin{corollary}
	\label{cor:10} 
	Assume that $\cl, D(\cl)\subset L^2(\rone)$ enjoys the properties 
	\begin{itemize}
		\item $\cl=\cl^*$, which maps real-valued into real-valued functions. 
		\item $\cl[\psi']=0$ for some smooth {\it real-valued} function $\psi$.
		\item  $\cl|_{\{\psi\}^\perp}\geq 0$.
	\end{itemize}
Then,  the eigenvalue problem \eqref{200} is spectrally stable. 
\end{corollary}

 \subsection{General instability criteria for KdV-like e-value problems}
 In the previous Section \ref{sec:2.4}, specifically Corollary \ref{cor:10}, we saw that a stability for the eigenvalue problem\eqref{200}, may be deduced directly to the properties $\cl[\psi']=0$ and $\cl|_{\{\psi\}^\perp}\geq 0$. We now would like to develop an alternative instability criteria, which essentially complements the one presented in Corollary \ref{cor:10}. 
 \begin{proposition}
 	\label{prop:43} 
 	Consider the eigenvalue problem \eqref{200}. Assume that 
 	\begin{enumerate}
 		\item $\cl$ maps real-valued to real valued functions and  $\cl[\psi']=0$, for some real-valued function $\psi$. 
 		\item $n(\cl)=1$, $\psi\perp Ker[\cl]$ 
 		\item $\cl$ has a ground state $\psi_0$. That is 
 		$$
 		-\si_0^2=\inf\{\dpr{\cl u}{u}: u\in D(\cl), \|u\|=1 \}<0
 		$$
 		and there is $\psi_0\in D(\cl)$, so that $\cl\psi_0=-\si_0^2 \psi_0$. 
 		\item The condition $\cl|_{\{\psi\}^\perp}\geq 0$ fails, and in fact the following infimum is achieved, i.e. 
 		$$
 		\mu_0=\inf\{\dpr{\cl u}{u}: u\in D(\cl), u\perp \psi, \|u\|=1 \}<0
 		$$
 		and there is a function $\Psi_0\perp \psi$, so that 
 		\begin{equation}
 			\label{e:40} 
 				\cl\Psi_0=\mu_0\Psi_0+\al \psi.
 		\end{equation}
 	\end{enumerate}
 Then, $\dpr{\cl^{-1}\psi}{\psi}>0$. 
 \end{proposition}
\begin{proof}
	We start with the case $\phi\perp \psi_0$. In this case,  since $\cl|_{\{\psi_0\}^\perp}\geq 0$ (recall $n(\cl)=1$, so only one negative e-value, $-\si_0^2$, is present) and $\psi\perp Ker[\cl]$ (in particular  $\cl^{-1}\psi$ is well-defined), we conclude  that $\psi$ belongs to the positive subspace of $\cl$, hence to the positive subspace of $\cl^{-1}$, whence  $\dpr{\cl^{-1}\psi}{\psi}>0$. 
	
	Assume now, $\dpr{\psi}{\psi_0}\neq 0$. In this case, by the properties of infimum taken over subspaces, $-\si_0^2\leq \mu_0<0$ and \eqref{e:40} holds for some $\Psi_0\perp \psi$. We claim that in fact $-\si_0^2<\mu_0$. Indeed, assume that $-\si_0^2=\mu_0$, and \eqref{e:40} holds. Then
	$$
	\cl \Psi_0=-\si_0^2 \Psi_0 +\al \psi
	$$
	Taking dot product with $\psi_0$ results in 
	$$
	-\si_0^2\dpr{\Psi_0}{\psi_0}=\dpr{\Psi_0}{\cl \psi_0}=\dpr{\cl\Psi_0}{\psi_0}=-\si_0^2 \dpr{\Psi_0}{\psi_0}+\al \dpr{\psi}{\psi_0}.
	$$
	Thus, since $\dpr{\psi}{\psi_0}\neq 0$, it follows that $\al=0$. But then, $\Psi_0$ is another eigenfunction of $\cl$, linearly independent from $\psi_0$, corresponding to the eigenvalue $-\si_0^2$, a contradiction with $n(\cl)=1$. 
	
	Thus, $-\si_0^2<\mu_0<0$ and \eqref{e:40} still holds for some $\al$. 
	We claim that $\al\neq 0$. Indeed, if we assume for a contradiction that $\al=0$, we have a second eigenfunction $\Psi_0$ (linearly independent from $\psi_0$, as $\Psi_0\perp \psi$, while $\dpr{\psi}{\psi_0}\neq 0$) corresponding to a negative eigenvalue $\mu_0$, in contradiction with $n(\cl)=1$. Thus, $\al\neq 0$, and then, we can infer from \eqref{e:40} that 
	$
	\Psi_0=\al (\cl-\mu_0)^{-1}\psi,
	$
	and so 
	$$
	0=\dpr{\Psi_0}{\psi}=\al \dpr{(\cl-\mu_0)^{-1}\psi}{\psi}
	$$
	whence $\dpr{(\cl-\mu_0)^{-1}\psi}{\psi}=0$. 
	Introduce the function 
	$$
	g(z)=\dpr{(\cl-z)^{-1} \psi}{\psi}, 
	$$
	over the spectrum-free interval $z\in [\mu_0, 0)$, which also well-defined at $z=0$. 
	Clearly, $g(\mu_0)=\dpr{(\cl-\mu_0)^{-1}\psi}{\psi}=0$, as established already, while 
	$$
	g'(z)=\dpr{(\cl-z)^{-2} \psi}{\psi}=\|(\cl-z)^{-1} \psi\|^2>0.
	$$
	Hence, 
	$$
	\dpr{\cl^{-1}\psi}{\psi}=g(0)=g(\mu_0)+\int_{\mu_0}^0 g'(z) dz = \int_{\mu_0}^0 g'(z) dz>0,
	$$
	as claimed. 
\end{proof}

   \section{Existence and stability of the waves  in the case of quadratic nonlinearity} 
     \label{sec:3}
   We  consider the inequality \eqref{20} for the case $p=3$ and arbitrary $\al>0$, 
   	\begin{eqnarray}
   	\label{22} 
   	\int_{-\infty}^\infty u^3(x) dx\leq C_{\al} \|u\|  \left(\|(D-1) u\|^2+ \al \|u\|^2\right). 
   \end{eqnarray}
   Here we take the value $C_\al$ to be the exact constant in \eqref{22}. In other words, 
   	\begin{equation}
   	\label{35} 
   C_\al:=\sup_{u\neq 0} I_\al[u]; \\ I_\al[u]:= \f{\int_{-\infty}^\infty u^3(x) dx}{\|u\|  \left(\|(D-1) u\|^2+ \al \|u\|^2\right) }.
\end{equation}
   \subsection{Existence of the waves}
   It is obviously not at all clear that a maximizer in \eqref{35} exists. This is the subject of the following proposition. 
   \begin{proposition}
   	\label{prop:10} 
   	For each $\al>0$, there exists a maximizer for \eqref{22}. That is, there exists a  function $\vp\in H^1(\rone)$, so that 
   $$
   		C_\al= \f{\int_{-\infty}^\infty \vp^3(x) dx}{\|\vp\|  \left(\|(D-1) \vp\|^2+ \al \|\vp\|^2\right)}.
   		$$
   \end{proposition}
   \begin{proof}
   	Note first that by testing $I$ with any smooth positive function,  $C_\al>0$. 
   	
   	To this end,  note that the functional $I$  in \eqref{35} is homogeneous, so we can  take a maximizing sequence $f_n$ with 
   	$\|(D-1) f_n\|^2+ \al \|f_n\|^2=1$.   It follows that $\|f_n\|\leq \f{1}{\al}$, and after eventually taking a subsequence, we may assume without loss of generality that $0\leq \lim_n \|f_n\|=a\leq \f{1}{\al}$.  
   	
   	Let us see first that the case $\lim_n \|f_n\|=0$ is contradictory. Indeed, suppose $\lim_n \|f_n\|=0$, we use the estimate \eqref{25} with $p=3$, to conclude 
   	$$
   	\int f_n^3 \leq M_\al \|f_n\|^{\f{5}{2}} (\|(D-1) f_n\|^2+ \al \|f_n\|^2)^{\f{1}{4}}=M_\al \|f_n\|^{\f{5}{2}}. 
   	$$
   	In such a case, 
   	$
   	  I[f_n]= \f{\int f_n^3}{\|f_n\|} \leq M_\al \|f_n\|^{\f{3}{2}}
   	$
   	whence 
   	$$
   	C_\al=\lim_n I[f_n]=\lim_n M_\al \|f_n\|^{\f{3}{2}}=0,
   	$$
   	a contradiction, since $C_\al>0$. 
   	Thus, by another rescaling,  we can  assume that $f_n\in H^1$, with 
   	\begin{equation}
   		\label{45} 
   			\lim_n \|f_n\|=1, \lim_n \|(D-1) f_n\|^2+ \al \|f_n\|^2=\la>0,
   	\end{equation}
   while $\{f_n\}$ is maximizing for \eqref{22} and so 
   \begin{equation}
   	\label{72} 
   	\lim_n \int f_n^3 dx = C_\al \la.
   \end{equation}
   	We apply the Lion's compensated compactness approach, see Theorem \ref{theo:102},  for the sequence 
   	\begin{equation}
   		\label{78} 
   			g_n =|(D-1) f_n|^2+ \al f_n^2\in L^1_+(\rone), \|g_n\|_{L^1}=\la.
   	\end{equation}
   	Our task is to refute the dichotomy and vanishing alternatives for $g_n$. 
   	\subsubsection{Splitting alternative does not hold} 
   	Assume first that splitting holds. Taking into account that the functional $I$ is translation invariant, i.e. $I[u(y+\cdot)]=I[u]$, we may replace $f_n\to f_n(y_n+\cdot)$ and so without loss of generality assume that $y_k=0$ in the splitting alternative. Also, passing to a subsequence allows us to represent 
   	\begin{equation}
   		\label{40} 
   			|(D-1) f_k|^2+ \al f_k^2 = g_{k,+}+g_{k,-}+e_k
   	\end{equation}
   where $support(g_{k,-})\subset (-\infty, -R_k), 	support(g_{k,+})\subset (R_k, +\infty)$, with some reals   $R_k: \lim_k R_k=+\infty$, so that 	for some $\mu\in (0, \la)$ and for all $\epsilon>0$ and all large enough $k>k_0(\epsilon)$, we have 
   \begin{equation}
   	\label{48} 
   	    |\int_{R_k}^{+\infty} g_{k,+}(x) dx - \mu|<\epsilon, 	|\int_{-\infty}^{-R_k} g_{k,-}(x) dx - (\la-\mu)|<\epsilon, \int |e_k(x)| dx<\epsilon.
   \end{equation}
   Introduce smooth functions $\chi_{\pm}$, so that $0\leq \chi_{\pm}(x)\leq 1$, 
   	$$
   	\chi_-(x)=\left\{
   	\begin{matrix}
   		1  & x<-1 \\
   		0 & x>-\f{1}{2}.
   		\end{matrix}
   	\right.\ \ \ \chi_+(x)=\left\{
   	\begin{matrix}
   		0 & x<\f{1}{2} \\
   		1 & x>1
   	\end{matrix}
   	\right.
   	$$
   	and $\chi_{n,\pm}(x):=\chi_\pm(x/R_n)$. That is, the functions $\chi_{n,-}$ restricts smoothly to the region $(-\infty, -R_n)$, and vanishes for $x>-\f{R_n}{2}$. Similar for $\chi_{n,+}$. 
   	
   	Multiplying \eqref{40} with $\chi_{n,+}^2$ results in 
   	\begin{equation}
   		\label{50} 
   		|\chi_{n,+} (D-1) f_n|^2+\al |\chi_{n,+} f_n|^2 = g_{n,+}+\chi_{n,+}^2 e_n.
   	\end{equation}
   	due to support considerations for $g_{n, \pm}$. Denoting $f_{n,+}:=\chi_{n,+} f_n$, we claim that for some absolute constant $C$ and for all $n$ large enough, 
   		\begin{equation}
   		\label{55} 
   	|\int 	|(D-1) f_{n,+}|^2+\al |f_{n,+}|^2  - \mu|\leq C(\epsilon+R_n^{-1}).
   	\end{equation}
   	Indeed, \eqref{50} is not far from this fact, except that there is a standard commutator term to be estimated. Specifically, by H\"older's and the equivalence $\|u\|_{H^1}\sim \|(D-1) u\|+\al \|u\|$, 
   	\begin{eqnarray*}
   	& & 	\int \left|	|(D-1) f_{n,+}|^2+\al |f_{n,+}|^2 - 	|\chi_{n,+} (D-1) f_n|^2-\al |\chi_{n,+} f_n|^2\right| dx=	\\
   	&=& \int |	|(D-1) f_{n,+}|^2 - 	|\chi_{n,+} (D-1) f_n|^2| dx \leq  \\
   		&\leq & (\| (D-1) f_{n,+}\| + \| (D-1) f_{n}\|)\| [D,\chi_{n,+}] f_n\|_{L^2} \leq \f{C}{R_n}  \|f_n\|_{H^1}^2  \leq  \f{C}{R_n}.  
   	\end{eqnarray*}
   	Based on \eqref{50}, \eqref{48}  and the last estimate, we conclude \eqref{55}. An identical argument establishes 
   		\begin{equation}
   		\label{60} 
   		\left|\int 	|(D-1) f_{n,-}|^2+\al |f_{n,-}|^2  - (\la-\mu)\right|\leq C(\epsilon+R_n^{-1}).
   	\end{equation}
   	We now estimate the value of $\int f_n^3$. We have 
   	\begin{equation}
   		\label{80}
   		\int f_n^3 = \int (f_n\chi_{n,+}+f_n\chi_{n,-}+f_n(1-\chi_{n,+}-\chi_{n,-}))^3 = 	\int f_{n,-}^3+\int f_{n,+}^3 + 
   		\int f_n^3 E_n(x) dx,
   	\end{equation}
   	where $|E_n(x)|\leq C ((1-\chi_{n,+})\chi_{n,-} + (1-\chi_{n,-})\chi_{n,+})$. Thus, 
 	\begin{equation}
 	\label{85}
   	|	\int f_n^3 E_n(x) dx|\leq C \|f_n\|^2 (\| f_n (1-\chi_{n,+})\chi_{n,-} \|_{L^\infty}+ \| f_n (1-\chi_{n,-})\chi_{n,+} \|_{L^\infty}).
\end{equation}
   	
   Next, we estimate $\| f_n (1-\chi_{n,+})\chi_{n,-} \|_{L^\infty}, \| f_n (1-\chi_{n,-})\chi_{n,+} \|_{L^\infty})$ appropriately, based on \eqref{48}.  Since they are similar, we will do only one of them, say 
   	$\| f_n (1-\chi_{n,-})\chi_{n,+} \|_{L^\infty})$. Denote $\zeta_n:=(1-\chi_{n,-})\chi_{n,+} $, so that $f_n \zeta_n$ is the object under consideration .  
   	
   	Multiply \eqref{40} by $\zeta_n^2$. By support considerations (note that the function $\zeta_n$ vanishes on the supports of 
   	$g_{n,  \pm}$), we have 
   	$$
   		|\zeta_n (D-1) f_n|^2+ \al (\zeta_n f_n)^2 = e_n \zeta_n^2. 
   	$$
   	Furthermore,  we have 
   	\begin{equation}
   		\label{70} 
   			| (D-1) (f_n \zeta_n)|^2+ \al (\zeta_n f_n)^2 =  e_n \zeta_n^2+ 	| (D-1) (f_n \zeta_n)|^2-|\zeta_n (D-1) f_n|^2
   	\end{equation}
   	Taking absolute values and integrating the last equality, and taking into account \eqref{48}, we obtain 
   	$$
   	\int 	| (D-1) (f_n \zeta_n)|^2+ \al (\zeta_n f_n)^2\leq C \epsilon + \int \left|(D-1) (f_n \zeta_n)|^2-|\zeta_n (D-1) f_n|^2\right| dx
   	$$
   	We had to deal with similar expression above, except the cutoff function was $\chi_{n,+}$ instead of $\zeta_n$. Same proof proceeds (as we have only used $\|\chi_n'\|_{L^\infty}\leq C R_n^{-1}$, which is also true for $\zeta_n$) with the same bound 
   	$$
   	\int \left|(D-1) (f_n \zeta_n)|^2-|\zeta_n (D-1) f_n|^2\right| dx\leq \frac{C}{R_n}.
   	$$
   	Going back to \eqref{70}, we record the estimate obtained herein 
   	\begin{equation}
   		\label{75} 
   			\int 	| (D-1) (f_n \zeta_n)|^2+ \al (\zeta_n f_n)^2\leq C(\epsilon+R_n^{-1}).
   	\end{equation}
   We now have by Sobolev embedding, 
   \begin{equation}
   	\label{285} 
   	  \|f_n\zeta_n\|_{L^\infty}^2 \leq C \|f_n\zeta_n\|_{H^1}^2\leq 
   	C	\int 	| (D-1) (f_n \zeta_n)|^2+ \al (\zeta_n f_n)^2\leq C(\epsilon+R_n^{-1}).
   \end{equation}
   	Taking into account \eqref{80} and \eqref{85}, we obtain 
   	\begin{equation}
   		\label{100} 
   			\int f_n^3 = \int f_{n,-}^3+ \int f_{n,+}^3+O(\epsilon+R_n^{-1}).
   	\end{equation}
   Furthermore, by passing to a subsequence, if necessary,  we may assume (due to $\|f_n\|=1$ and support considerations), 
    that there are $a_{\pm}\geq 0$, so that 
    \begin{equation}
    	\label{105} 
    	    \lim_n \|f_{n, \pm}\|=a_\pm, a_-^2+a_+^2\leq 1. 
    \end{equation}
    We now apply the estimate \eqref{22} for $ \int f_{n,\pm}^3$. We have 
    \begin{eqnarray*}
    \int f_{n,-}^3+ \int f_{n,+}^3 &\leq &   C_\al(\|f_{n,-}\|(\|(D-1)f_{n,-}\|^2+\al \|f_{n,-}\|^2)+\|f_{n,+}\|(\|(D-1)f_{n,+}\|^2+\al \|f_{n,+}\|^2)\leq \\
    	 &\leq & C_\al  ((\la-\mu) \|f_{n,-}\|+ \mu \|f_{n,+}\|)+ C(\epsilon+R_n^{-1}),
    \end{eqnarray*}
   	where we have used  the relations \eqref{55} and \eqref{60}. Combing this with \eqref{100} yields the bound 
   	$$
   		\int f_n^3\leq C_\al  ((\la-\mu) \|f_{n,-}\|+\mu  \|f_{n,+}\|)+ C(\epsilon+R_n^{-1}). 
   	$$
   	Now, take $\lim_n$. By \eqref{72}, \eqref{105} and Cauchy-Schwartz, we have 
   	$$
   	C_\al \la\leq C_\al ((\la-\mu)a_-+\mu a_+)+C\epsilon\leq C_\al\sqrt{a_-^2+a_+^2}\sqrt{(\la-\mu)^2+\mu^2} +C\epsilon\leq C_\al \sqrt{(\la-\mu)^2+\mu^2} +C\epsilon. 
   	$$
   	This is clearly contradictory, for sufficiently small $\epsilon$,  as $\sqrt{(\la-\mu)^2+\mu^2}<\la$, due to $\mu\in (0, \la)$. This shows that the splitting alternative does not hold. 
   	
   	\subsubsection{Vanishing does not hold}
   	Vanishing for $g_n$ is in fact  much simpler to refute. Indeed, assume that it holds for say $R=1$. Then, 
   	$$
   	\lim_k \left(  \sup_{y\in\rone} \int_{y-1}^{y+1} |(D-1)f_{n_k}|^2+\al |f_{n_k}|^2 dx    \right)=0.
   	$$
   	In particular, it follows that 
   	\begin{equation}
   		\label{110} 
   		\lim_k \left(  \sup_{y\in\rone} \int_{y-1}^{y+1}  |f_{n_k}|^2 dx    \right)=0.
   	\end{equation}
   	Recall the  cut-off function $\eta\in C^\infty_0(\rone)$, so that $0\leq \eta<1, \eta(x)=1, |x|<1, \eta(x)=0, |x|>2$. We have for each integer  $j$, 
   	\begin{eqnarray*}
   	& & 	\int_{j-1}^{j+1} f_{n_k}^3(x) dx  \leq 	\int  |f_{n_k}(x) \eta(j-x)|^3 dx\leq \\
   		& \leq & 
   		C_\al \|f_{n_k}\eta(j-\cdot)\|(\|(D-1)f_{n_k}\eta(j-\cdot)\|^2+\al \|f_{n_k}\eta(j-\cdot)\|^2)\leq C \|f_{n_k}\eta(j-\cdot)\| \| f_{n_k}\eta(j-\cdot)\|_{H^1}^2.
   	\end{eqnarray*}
   	where we have 
   	applied \eqref{30}with $p=3$, and the equivalence $\|(D-1)u\|^2+\al \|u\|^2\sim \|u\|_{H^1}^2$. Fixing an arbitrary $\epsilon>0$, we can find, as a consequence of \eqref{110}, a  large $k_0$, so that $\|f_{n_k}\eta(j-\cdot)\|<\epsilon$ for all $k>k_0$ and for all $j$. Adding up the last estimate in integer $j$, we obtain 
   	$$
   	2 \int_{-\infty}^{+\infty} f_{n_k}^3(x) dx\leq C \epsilon \sum_{j=-\infty}^{+\infty} \| f_{n_k}\eta(j-\cdot)\|_{H^1}^2\leq C \epsilon \|f_{n_k}\|_{H^1}^2\leq C \epsilon,
   	$$
   	where we have used the orthogonality in the $j$ sum and $1=\|g_{n_k}\|_{L^1}\sim  \|f_{n_k}\|_{H^1}^2$. This clearly contradicts the setup for appropriately small  $\epsilon$, since $\lim_k \int_{-\infty}^{+\infty} f_{n_k}^3(x) dx=C_\al \la$. 
   	\end{proof}
   
   \subsubsection{Completion of the proof of Proposition \ref{prop:10}}
   As we have shown that vanishing and splitting are not viable alternatives, compactness/tightness holds. After taking translations and passing to a subsequence, we may assume that for every $\epsilon>0$, there exists $R=R(\eps)$ and $k_0=k_0(\eps)$, 
   so that for all $k\geq k_0$, we have 
   $$
   \int_{-R}^R |(D-1) f_{n_k}|^2+\al |f_{n_k}|^2 dx >\la-\eps.
   $$
   It follows that there is a further subsequence, which {\it strongly} converges in $H^1(\rone)$ to a limit $\vp$, so that $ \int_{-R}^R |(D-1) \vp|^2+\al \vp^2 dx=\la$, $\|\vp\|=1$. As a consequence of the Sobolev embedding $H^1(\rone)\hookrightarrow L^3(\rone)$, this further sequence converges strongly to $\vp$ in $L^3$ as well, whence $\vp$ is a maximizer of the inequality \eqref{22}. With this, the proof of Proposition \ref{prop:10} is complete. 
   
   Next, we derive properties of the function $\al\to C_\al$. 
   \subsubsection{Properties of the map $\al\to C_\al$}
   \begin{proposition}
   	\label{prop:35} 
   	The map $\al\to C_\al$, defined on $\rone_+$ is a continuous, strictly  decreasing function. In addition 
   	\begin{equation}
   		\label{320} 
   			\lim_{\al\to 0+} C_\al=+\infty, \lim_{\al\to +\infty} C_\al=0.
   	\end{equation}
   \end{proposition}
\begin{proof}
	First, since for $0<\al_1<\al_2$, we have $I_{\al_1}[u]\geq I_{\al_2}[u]$, we clearly have that $C_{\al_1}\geq C_{\al_2}$. Denoting the maximizer $\vp_\al$ for \eqref{35}, we have that 
	$$
	C_{\al_2}=\f{\int \vp_{\al_2}^3 }{\|\vp_{\al_2}\| (\|(D-1)\vp_{\al_2}\|^2+\al_2\|\vp_{\al_2}\|^2)}< \f{\int \vp_{\al_2}^3 }{\|\vp_{\al_2}\| (\|(D-1)\vp_{\al_2}\|^2+\al_1\|\vp_{\al_2}\|^2)}\leq C_{\al_1},
	$$
	thus establishing the strict decreasing. 
	
	Next, we tackle the continuity. As monotone function, we confirm the existence of left and right limits, and we have to rule out jumps. To that end, fix $\al>0$ and a sequence $\al_n\to \al-$. Consider the corresponding maximizers $\vp_{\al_n}$, with the extra stipulation that $\|\vp_{\al_n}\|=1$ (which is possible via a rescaling), and  $\|(D-1)\vp_{\al_n}\|^2+\al_n \|\vp_{\al_n}\|^2=\la_n\to \la$. Note $\la_n>\al_n$, and bounded away from zero. We have that 
	\begin{eqnarray*}
		C_{\al_n} &=& \f{\int \vp_{\al_n}^3 }{\|\vp_{\al_n}\| (\|(D-1)\vp_{\al_n}\|^2+\al_n\|\vp_{\al_n}\|^2)} = \f{\int \vp_{\al_n}^3 }{\|\vp_{\al_n}\| (\|(D-1)\vp_{\al_n}\|^2+\al\|\vp_{\al_n}\|^2)+(\al_n-\al)}=\\
		&=& \f{\int \vp_{\al_n}^3 }{\|\vp_{\al_n}\| (\|(D-1)\vp_{\al_n}\|^2+\al\|\vp_{\al_n}\|^2)}+(\al-\al_n) \f{\int \vp_{\al_n}^3 }{\la_n^2}+O((\al-\al_n)^2)\leq C_{\al}+ D(\al-\al_n).
	\end{eqnarray*}
Taking $\lim_n$ on both sides leads to $C_{\al-}\leq C_\al$. But the function $\al\to C_\al$ is decreasing, whence we get the opposite inequality, so $C_{\al-}=C_\al$ for all $\al>0$. 

Regarding the right limits, take $\al_n\to \al+$. We have 
$$
C_{\al+}=\lim_n C_{\al_n}\geq \lim_n \f{\int \vp_{\al}^3 }{\|\vp_{\al}\| (\|(D-1)\vp_{\al}\|^2+\al_n\|\vp_{\al}\|^2)}=\f{\int \vp_{\al}^3 }{\|\vp_{\al}\| (\|(D-1)\vp_{\al}\|^2+\al \|\vp_{\al}\|^2)}=C_\al.
$$
Again, by the decreasing  of the map, we gain the opposite inequality $C_{\al+}\leq C_\al$. Altogther, 
$$
C_{\al-}=C_\al=C_{\al+},
$$
which is the continuity. 

We now turn our attention to the asymptotics \eqref{320}. Tothis end, introduce a cut-off function $\eta\in C^\infty_0(\rone)$, so that $0\leq \eta<1, \eta(x)=1, |x|<1, \eta(x)=0, |x|>2$. Fix a small $0<\epsilon<<1$ and  consider a test function  
\begin{equation}
	\label{d:18}
	\widehat{u}_\eps(\xi)=
	\eta\left(\f{\xi-1}{\eps}\right)+\eta\left(\f{\xi+1}{\eps}\right)+\eps \eta\left(\f{\xi}{\eps}\right).
\end{equation}
 An elementary calculation (note the disjoint support of the three  pieces, defining $\hat{u}_\eps$) shows that 
$$
	\|(D-1) u_\eps\|^2=\int (|\xi|-1)^2 
	\left(\eta^2\left(\f{\xi-1}{\eps}\right) +\eta^2\left(\f{\xi+1}{\eps}\right)+\eps^2 \eta^2\left(\f{\xi}{\eps}\right)\right)d\xi \sim \eps^3, \ \ 
	 \|u_\eps\|\sim \sqrt{\eps}. 
$$
On the other hand, using elementary properties of the Fourier transform 
$$
u_\eps(x)=\eps \check{\eta}(\eps x)(e^{ i x}+e^{- i x}+\eps).
$$
It follows that $u^3_\eps(x)=\eps^3 \check{\eta}^3(\eps x)
(e^{ i x}+e^{- i x}+\eps)^3$.  Expanding the cubic term yields 
$$
(e^{ i x}+e^{-i x}+\eps)^3=6 \eps + \eps^3+ A_\pm  e^{\pm   i x}+B_\pm e^{\pm 2 i x}+C_\pm  e^{\pm 3 i x}, 
$$
for some constants $A_\pm, B_\pm, C_\pm$. Thus,  we obtained  the relation 
$$
\int u^3_\eps =6 \eps^4 \int \check{\eta}^3(\eps x) +\eps^6 \int \check{\eta}^3(\eps x) + A_\pm \int e^{\pm i x}\check{\eta}^3(\eps x) +B_\pm \int e^{\pm 2 i x}\check{\eta}^3(\eps x) +C_\pm \int e^{\pm 3 i x}\check{\eta}^3(\eps x).
$$
Note that the last three terms are of order $O(\eps^N)$ for all integer $N$, whence 
$$
\int u^3_\eps(x)=6 \eps^4 \int \check{\eta}^3(\eps x)dx+O(\eps^5)=
6 \eps^3 \int \check{\eta}^3(x)dx+O(\eps^5).
$$
By Plancherel's $ \int \check{\eta}^3(x)dx=\int \eta(x) (\eta*\eta)(x) dx>0$. Thus, for all $\al>0$ and small enough $\eps$, we have 
$$
C_\al\geq I_\al[u_\eps]\geq C \f{\eps^3}{\sqrt{\eps}(\eps^3+\al \eps)}.
$$
Thus, 
$$
\lim_{\al\to 0+} C_\al\geq \f{C}{\sqrt{\eps}},
$$
for all $\eps>0$.  Thus, we infer the asymptotic $\lim_{\al\to 0+} C_\al=+\infty$.  

Regarding the claim $\lim_{\al\to +\infty} C_\al=0$ in \eqref{320}, we provide a direct, Fourier proof  of the estimate \eqref{30}, which provides a better tracking of the constants. Indeed, by Sobolev embedding, H\"older's and Plancherel's, we have 
\begin{eqnarray*}
	\int \vp^3(x) dx \leq \|\vp\|_{L^3}^3\leq C \|D^{\f{1}{6}}\vp\|^3\sim C\left(\int |\xi|^{\f{1}{3}} |\hat{\vp}(\xi)|^2 d\xi\right)^{\f{3}{2}}\leq C \|\vp\|  \left(\int |\xi|^{\f{1}{2}} |\hat{\vp}(\xi)|^2 d\xi \right).
\end{eqnarray*}
In addition,  we have by Young's   inequality, 
$
 |\xi|^{\f{1}{2}} \leq C\al^{-3/4} ((|\xi|-1)^2+\al), 
$
for  $\al\geq 1$ and some absolute constant $C$.  Thus, 
$$
\int \vp^3(x) dx  \leq C \al^{-3/4} \|\vp\|  \left(\int  ((|\xi|-1)^2+\al)  |\hat{\vp}(\xi)|^2 d\xi \right)=C \al^{-3/4} \|\vp\|  (\|(D-1)\vp\|^2+\al \|\vp\|^2).
$$
This shows that $C_\al \leq C\al^{-3/4}$ for large $\al$, whence $\lim_{\al\to +\infty} C_\al=0$. 
\end{proof}

   \subsection{Stability of the waves}
   Now that we have constructed the wave $\vp$, we can derive the several  useful spectral properties. We collect these in the following proposition. 
   \begin{proposition}
   	\label{prop:20} 
   	For the maximizer  $\vp$ of the interpolation inequality, we have the following properties 
   	\begin{itemize}
   		\item it satisfies the Euler-Lagrange equation 
   		\begin{equation}
   			\label{115} 
   			(D-1)^2 \vp+\left(\al+\f{\|(D-1)\vp\|^2+\al \|\vp\|^2}{2\|\vp\|^2}\right) \vp - \f{3}{2C_\al\|\vp\|} \vp^2=0.
   		\end{equation}
   	
   	\item The linearized operator 
   	$$
   	\cm_+:=(D-1)^2  + \om_\al - \f{3}{C_\al\|\vp\|}\vp,  \ \ 
   	\om_\al=\al+\f{\|(D-1)\vp\|^2+\al \|\vp\|^2}{2\|\vp\|^2} 
   	$$
   	has the property $\cm_+|_{\{\vp\}^\perp}\geq 0$. 
   	\end{itemize}
   \end{proposition}
   \begin{proof}
   	Since $\vp$ is a maximizer of \eqref{22}, fix a test function $h\in H^2(\rone)$. We have that 
   	$$
   	g(\eps):= \f{\|\vp+\eps h\|}{\|\vp\|}(\|(D-1)(\vp+\eps h)\|^2+\al \|\vp+\eps h\|^2) 
   	- \f{1}{C_\al \|\vp\|} \int (\vp+\eps h)^3 dx
   	$$
   	has an absolute  minimum $g(0)=0$. Thus, $g'(0)=0$, whence we derive the Euler-Lagrange equation \eqref{115}. 
   	Additionally, we have that $g''(0)\geq 0$. In order to simplify the calculations (and in view of our goal of showing that  $\cm_+|_{\{\vp\}^\perp}\geq 0$), select in addition the increment $h$, so that  $h\perp \vp$. We obtain the quadratic form in $h$
   	$$
   	2 g''(0)= \dpr{((D-1)^2+\om - \f{3}{C_\al \|\vp\|} \vp)h}{h}.
   	$$
   	In other words, $\dpr{\cm_+ h}{h}\geq 0$ for all $h: h\perp \vp$, which is the desired conclusion.

   \end{proof}
   
As we have mentioned above, the functional $I[u]$ is homogeneous of degree zero, and as a consequence, it has many maximizers, all  in the form $a\vp$ for arbitrary $a\neq 0$.   In particular, the transformation 
\begin{equation}
	\label{298} 
	\phi:=\f{3}{2C_\al \|\vp\|} \vp, \|\phi_\al\|=\f{3}{2C_\al}
\end{equation}
   yields another maximizer of \eqref{22}. Plugging in this formula in the Euler-Lagrange equation \eqref{115}, results in  the relation 
  \begin{equation}
  	\label{300} 
  	   	(D-1)^2 \phi +\om_\al  \phi - \phi^2=0,
  \end{equation}
which is of course \eqref{14}, with the specific dependence on $\al$. 
Note that $\phi$ also depends on $\al$ and moreover,  $\|\phi\|=\f{3}{2 C_\al}$. Based on the way $\phi$ is obtained, namely by rescaling  of $\vp$, we may infer the property $\cl_+|_{\{\phi\}^\perp}\geq 0$ for the corresponding linearized operator 
\begin{equation}
	\label{310} 
	\cl_+=(D-1)^2  + \om_\al - 2\phi,
\end{equation}
based on the same property for the related  operator $\cm_+$.

We are now ready for the stability statement of the traveling wave $\phi$. 
\begin{proposition}
	\label{prop:40} 
	Let $\al>0$ and $\phi_\al$ be a solution of \eqref{300}, obtained as a rescale of the maximizer $\vp$ (through \eqref{298}), which was constructed in Proposition \ref{prop:10}. Then, $\phi$ is a spectrally stable traveling wave solution, in the sense of Definition \ref{defi:10}. 
\end{proposition}
  \begin{proof}
  	For the proof, we simply invoke the property \eqref{310} just established for the linearized operator $\cl_+$ and Corollary \ref{cor:10}. 
  \end{proof}

   \subsection{Proof of Theorem \ref{theo:10}}
   
   We have all the ingredients for the proof of Theorem \ref{theo:10}. Indeed, the waves $\vp_\al, \al>0$ are constructed in Propostion \ref{prop:10}. By the rescaling \eqref{298}, we obtain the waves $\phi_\al$, which are actual traveling wave solutions of the Benjamin equation. By \eqref{298}, $\al\to \|\phi_\al\|=\f{3}{2C_\al}$, which is a continuous, strongly increasing function, according to Proposition \eqref{prop:35}. Moreover, 
   $$
   \lim_{\al\to 0+}\|\phi_\al\|=0,  \lim_{\al\to \infty}\|\phi_\al\|=+\infty. 
      $$
   Finally, by Proposition \ref{prop:40}, the steady solutions $\phi$ are spectrally stable solutions  of \eqref{b:10}. 
   
   \subsection{Proof of Proposition \ref{prop:100}}
   Clearly, due to the formula \eqref{c:10}, the mapping $\om\to \om_\al$ is $C^1$. Moreover, we can take a derivative with respect to the parameter $\al$ in the profile equation \eqref{14}. As a result, we obtain the formula
   \begin{equation}
   	\label{c:20} 
   	\cl_+[\p_\al \phi]=-\om'(\al) \phi_\al.
   \end{equation}
   So, either $\om'(\al)=0$ and we are done, or $\om'(\al)\neq 0$. In such  a case,  $\phi_\al\in Im[\cl_+]$, whence $\phi_\al\perp  \ker[\cl_+]$ and  so $\cl_+^{-1}\phi_\al$ is well-defined. Applying $\cl_+^{-1}$ in \eqref{c:20} yields 
   $$
   \p_\al \phi = -\om'(\al) \cl_+^{-1}[\phi].
   $$
   Taking dot product with $\phi$  results in 
   \begin{equation}
   	\label{c:30} 
   	  \f{1}{2} \p_\al \|\phi_\al\|^2=-\om'(\al)\dpr{\cl_+^{-1}\phi}{\phi}.
   \end{equation}
 Since by Theorem \ref{theo:10}, $\al\to \|\phi_\al\|$ is strictly increasing, it follows that the left-hand side of \eqref{c:30} is non-negative. Per the statement of Proposition \ref{prop:100},     suppose $\p_\al \|\phi_\al\|^2>0$. It will suffice to show 
 that $\dpr{\cl_+^{-1}\phi}{\phi}\leq 0$. This actually follows from a general lemma, which appears verbatim  in \cite{PS18}, but it has its origins in some  earlier results, \cite{Pel}. 
 \begin{lemma}[\cite{PS18}]
 	\label{le:93} 
 	Suppose that $\ch$ is a self-adjoint operator on a Hilbert space $X$, so that $\ch|_{\{\xi_0\}^\perp}\geq 0$. Next, assume $\xi_0\perp Ker[\ch]$, so that $\ch^{-1} \xi_0$ is well-defined. 
 	Finally, assume $\dpr{\ch \xi_0}{\xi_0}\leq 0$. Then
 	$$
 	\dpr{\ch^{-1} \xi_0}{\xi_0}\leq 0.
 	$$ 
 \end{lemma}
 We provide a short proof of this result in the Appendix for completeness, but let us apply it in our situation. Indeed, 
 setting $\xi_0:=\phi$, and since we know $\phi\perp \ker[\cl_+]$, we need to check the sign of $\dpr{\cl_+ \phi}{\phi}$. By direct calculation 
 $$
 \dpr{\cl_+ \phi}{\phi}=\dpr{((D-1)^2+\om - 2\phi)\phi}{\phi}=-\int\phi^2<0,
 $$
   as required. Thus, Lemma \ref{le:93} applies and 
   $\dpr{\cl_+^{-1} \phi}{\phi}\leq 0$. Since $\p_\al \|\phi_\al\|^2>0$ by assumption, it follows that in fact $\dpr{\cl_+^{-1} \phi}{\phi}<0$, whence 
   $$
   \om'(\al)=-\f{\p_\al \|\phi_\al\|^2}{2\dpr{\cl_+^{-1} \phi}{\phi}}>0.
   $$
   
 
   \section{Existence and stability of the waves  in the case of general power nonlinearity} 
   In this section, we tackle the problem of existence of the waves for the problem with general power nonlinearity, i.e. \eqref{b:50}. Incidentally, there are two methods for constructing such solutions, which work in different ranges. We start with the one which is a direct generalization of the approach in Section \ref{sec:3}. 
   \subsection{Waves constructed as maximizers of an interpolation inequality}
   
    \label{sec:4.1}
   Clearly, as \eqref{30} holds true for $p\in (2,6]$, there is a maximal (exact) constant for it, let us denote it by $C_{\al, p}$. That is, 
   $$
   C_{\al,p}= \sup_{f\neq 0: f\in H^1} 
   \f{\int_{-\infty}^\infty |f(x)|^p dx}{\|f\|^{p-2}  \left(\|(D-1) f\|^2+ \al \|f\|^2\right)}, 2<p\leq 6, \al>0.
   $$
   \begin{proposition}
   	\label{prop:80} 
   	Let $2<p\leq 6$.  For each $\al>0$, there exists a maximizer for \eqref{30}. That is, there exists a  function $\vp=\vp_{p,\al} \in H^1(\rone)$, so that 
   	$$
   	C_{\al,p}= \f{\int_{-\infty}^\infty |\vp(x)|^p dx}{\|\vp\|^{p-2} \left(\|(D-1) \vp\|^2+ \al \|\vp\|^2\right)}.
   	$$
   	In addition, the function satisfies the Euler-Lagrange equation 
   		\begin{equation}
   		\label{215} 
   		(D-1)^2 \vp+\left(\al+\f{p-2}{2\|\vp\|^2}(\|(D-1)\vp\|^2+\al \|\vp\|^2)\right) \vp -\f{p}{2 C_{\al,p}\|\vp\|^{p-2}} |\vp|^{p-2}\vp=0.
   	\end{equation}
   The linearized operator 
   $$
   \cm_+=(D-1)^2  + \om_{\al,p} - \f{p(p-1)}{2 C_{\al,p}\|\vp\|^{p-2}} |\vp|^{p-2},  \ \om_{\al,p}=\f{p \al}{2} +\f{(p-2) }{2} \f{\|(D-1)\vp\|^2}{\|\vp\|^2} 
   $$
   has the property $\cm_+|_{\{\vp\}^\perp}\geq 0$. 
   \end{proposition}
   \begin{proof}
   	The proof consists of mostly a reprise of Proposition \ref{prop:10} for the existence part, as well as Proposition \ref{prop:20} for the Euler-Lagrange equation and the spectral properties of the linearized operator. 
   	
   	Let us begin with the existence, and we use the same notations as in Proposition \ref{prop:10}. We prove in an identical manner that $C_{\al, p}>0$ as well as the statement that any maximizing sequence for \eqref{30}, with the property $\|(D-1) f_n\|^2+\al \|f_n\|^2=1$,  cannot have $\lim_n \|f_n\|=0$. This leads to the assumption \eqref{45} for a maximizing sequence 
   	$\{f_n\}\subset H^1$. So, 
   	\begin{equation}
   		\label{500}
   		\lim_n 	\int |f_n(x)|^p dx = C_{\al,p} \la.
   	\end{equation}
   We then apply the Lions's compensated compactness result for 
   $
   g_n=|(D-1)f_n|^2+\al f_n^2
   $
    whence \eqref{78}. Next, we proceed to refute the dichotomy and vanishing for $g_n$. 
    \subsubsection{Splitting does not hold}
    To this end, we may again assume without loss of generality, 
    that $y_n=0$ in the original splitting alternative. Next, we  write \eqref{40} and \eqref{48} for some $\mu\in (0, \la)$, and introduce the same functions $\ch_\pm$, and analogously $\chi_{n, \pm}(x)=\chi_\pm(x/R_n)$. Again\footnote{There is no difference in the object considered herein, as this is all about the functions $g_n$, which are defined the same as in Section \ref{sec:3}} multiplying \eqref{40} by $\chi_{n, +}$ results in \eqref{50}, whence \eqref{55} and \eqref{60}. 
    
    Next, we estimate the value of $\int |f_n|^p$, which is a bit different than Section \ref{sec:3}. We have 
    \begin{eqnarray*}
    		\int |f_n|^p &=&  \int  |f_n\chi_{n,+}|^p+ \int  |f_n\chi_{n,-}|^p+\int |f_n|^p(1-\chi_{n,+}^p-\chi_{n,-}^p) dx = \\
    		&=& \int  |f_{n,+}|^p+\int  |f_{n,-}|^p+\int |f_n|^p\zeta_n 
    \end{eqnarray*}
    where $\zeta_n=\zeta(x/R_n)$, $\zeta:=1-\chi_{+}^p-\chi_{-}^p\in C^1(\rone)$, $supp\  \zeta\subset (-1,1)$. 
    
    Similar to \eqref{105}, since $\|f_n\|=1$, and up to taking subsequences, 
    we may assume that there are $a_\pm\geq 0$, so that 
    \begin{equation}
    	\label{605} 
    	\lim_n \|f_n\|=a_\pm: a_+^2+a_-^2\leq 1.
    \end{equation}
    Applying \eqref{30} for $f_{n,\pm}$ (with the exact constant $C_{\al,p}$ as defined above) and the bounds \eqref{55} and \eqref{60},  we obtain 
    \begin{eqnarray*}
    & & 	\int |f_{n,-}|^p+ \int |f_{n,+}|^p \leq \\
    & \leq &  C_{\al,p}(\|f_{n,-}\|^{p-2}(\|(D-1)f_{n,-}\|^2+\al \|f_{n,-}\|^2)+\|f_{n,+}\|^{p-2} 
    	(\|(D-1)f_{n,+}\|^2+\al \|f_{n,+}\|^2)\leq \\
    	&\leq & C_{\al,p} ((\la-\mu) \|f_{n,-}\|^{p-2}+ \mu \|f_{n,+}\|^{p-2})+ C(\epsilon+R_n^{-1}).
    \end{eqnarray*}
For the error term, $\int |f_n|^p\zeta_n$, we employ the H\"older's bounds 
\begin{eqnarray}
	\label{610}
& & 	\int |f_n|^p\zeta_n \leq \|f_n\zeta_n\|_{L^\infty} \int |f_n|^{p-1}, p\geq 3 \\
	\label{615}
& & 	\int |f_n|^p\zeta_n \leq \|f_n\zeta_n\|_{L^q} 
\|f_n\|^{p-1}=\|f_n\zeta_n\|_{L^q}, 2<p< 3, \f{1}{q}+\f{p-1}{2}=1.
\end{eqnarray}
In the case, $p\geq 3$, we can further control by Sobolev embedding and \eqref{45}, 
$$
 \int |f_n|^{p-1}\leq C\|f_n\|_{H^1}^{p-1}\sim (\|(D-1)f_n\|^2+\al \|f_n\|^2)^{\f{p-1}{2}}\leq C.
$$
In the case, $p\in (2,3)$, we apply Gagliardo-Nirenberg's, noting that  $q=\f{2}{3-p}>2$, as $p>2$. Specifically, 
$$
\|f_n\zeta_n\|_{L^q}\leq \|f_n\zeta_n\|_{L^\infty}^{1-\f{2}{q}} \|f_n\|_{L^2}^{\f{2}{q}}= \|f_n\zeta_n\|_{L^\infty}^{1-\f{2}{q}}. 
$$
Combining all these observations,   we see that in order to justify  the  error term 
\begin{equation}
	\label{630} 
	\lim_n 	\int |f_n|^p\zeta_n=0,
\end{equation}
it will suffice to show $\lim_n \|f_n\zeta_n\|_{L^\infty}=0$. This was however already established in Section \ref{sec:3} in \eqref{285} - in there, the function was taken slightly different (namely $\zeta=(1-\chi_+)\chi_-$), but the result holds because of the property $\zeta\in C^1, \ supp \ \zeta \subset (-1,1)$, which is valid here as well. 

 Putting it all together, we have after taking limits, 
\begin{eqnarray*}
C_{\al, p}\la &=& \lim_n \int |f_n|^p\leq C_{\al,p} ((\la-\mu) a_-^{p-2}+ \mu a_+^{p-2})+ C \epsilon+\limsup_n \int |f_n|^p\zeta_n= \\
&=& C_{\al,p} ((\la-\mu) a_-^{p-2}+ \mu a_+^{p-2})+ C \epsilon, 
\end{eqnarray*}
whence for all $\epsilon>0$, we must have the inequality
\begin{equation}
	\label{650} 
	C_{\al, p}\la \leq C_{\al,p} ((\la-\mu) a_-^{p-2}+ \mu a_+^{p-2})+ C \epsilon. 
\end{equation}
Recall that as a consequence of \eqref{605}, we have $a_\pm \in [0,1]$. Then, it is easy to see that for $\mu\in (0, \la)$,  
$$
(\la-\mu) a_-^{p-2}+ \mu a_+^{p-2}<\la.
$$
Indeed, for $a_-=0$, we use that $ \mu a_+^{p-2}\leq \mu<\la$, and similarly for $a_+=0$. When both $a_\pm>0$, we see that by \eqref{605}, $a_\pm<1$, whence $(\la-\mu) a_-^{p-2}+ \mu a_+^{p-2}<\la-\mu+\mu=\la$, for a strict inequality. 
All in all, we obtain a contradcition in \eqref{650}, the moment we select $\epsilon$, so that 
$$
C\epsilon= \f{C_{\al,p}}{2} [\la- ((\la-\mu) a_-^{p-2}+ \mu a_+^{p-2})].
$$
\subsubsection{Vanishing does not hold}
The proof here is identical as in Section \ref{sec:3}. Indeed, \eqref{110} follows as in there. Then, we apply \eqref{30} for $p\in (2,6]$ for the subsequence $f_{n_k}$ (defined in \eqref{110})
We have for each integer\footnote{Recall that the  cut-off function $\eta\in C^\infty_0(\rone)$ satisfies  $0\leq \eta<1, \eta(x)=1, |x|<1, \eta(x)=0, |x|>2$. }  $j$, 
\begin{eqnarray*}
 	\int_{j-1}^{j+1} f_{n_k}^p(x) dx  & \leq &  	\int  |f_{n_k}(x) \eta(j-x)|^p dx\leq \\
	&\leq & 
	C_\al \|f_{n_k}\eta(j-\cdot)\|^{p-2} (\|(D-1)f_{n_k}\eta(j-\cdot)\|^2+\al \|f_{n_k}\eta(j-\cdot)\|^2) \\
	&\leq &   C \|f_{n_k}\eta(j-\cdot)\|^{p-2}
	\| f_{n_k}\eta(j-\cdot)\|_{H^1}^2.
\end{eqnarray*}
Based on \eqref{110}, fixing an arbitrary $\eps$, we will be able to find $k_0$, so that $\|f_{n_k}\eta(j-\cdot)\|<\eps$, whence adding up the last estimate in integer $j$, we obtain 
$$
2\int_{-\infty}^{+\infty} |f_{n_k}|^p\leq C \epsilon^{p-2} \sum_{j=-\infty}^{+\infty} \| f_{n_k}\eta(j-\cdot)\|_{H^1}^2\leq C \epsilon^{p-2} \|f_{n_k}\|_{H^1}^2\leq C \epsilon^{p-2}.
$$
For small enough $\eps$, this contradicts $\lim_k \int_{-\infty}^{+\infty} |f_{n_k}|^p=C_{\al,p}\la$, so vaishing does not hold either. 

\subsubsection{The Euler-Lagrange equation and the second variation}
The proof here proceeds similarly to before. Fix a test function $h$ and  consider the scalar function  
$$
g(\eps):= \f{\|\vp+\eps h\|^{p-2}}{\|\vp\|^{p-2}}(\|(D-1)(\vp+\eps h)\|^2+\al \|\vp+\eps h\|^2) 
- \f{1}{C_{\al,p}\|\vp\|^{p-2}} \int |\vp+\eps h|^p dx.
$$
As $\vp$ is a minimizer, we have that $0$ is an absolute minimum for the function $g: g(0)=0$, and as a consequence $g'(0)=0$, while $g''(0)\geq 0$. Writing this out, and taking into account that this should hold for arbitrary test function $h$, we establish that $\vp$ is a weak solution of the  Euler-Lagrange equation\footnote{ and hence strong and in fact classical $C^\infty(\rone)$ solution by standard elliptic theory} \eqref{215}.  Furthermore, we take as before,  a test function $h: h\perp \vp$. This simplifies the expression for $g''(0)$ quite a bit, and as a corollary, we obtain the property $\cm_+|_{\{\vp\}^\perp}\geq 0$ for the second variation. 
   \end{proof}
   \subsubsection{The waves $\phi_{\al,p}$}
   Similar to the approach in Section \ref{sec:3}, specifically \eqref{298}, we define 
   \begin{equation}
   	\label{398} 
   	\phi_{\al,p}:=\left(\f{p}{2 C_{\al,p}\|\vp\|^{p-2}}\right)^{\f{1}{p-2}} \vp_{\al,p}
   \end{equation}
 Once again, 
   $$
 \|\phi_{\al,p}\|=\left(\f{p}{2 C_{\al,p}}\right)^{\f{1}{p-2}}
	$$
whence we conclude that the  function $\al\to \|\phi_{\al,p}\|$ is increasing, as long as we can establish $\al\to C_{\al,p}$ is 
decreasing. This is done with an identical argument  to Proposition \ref{prop:35}. In fact, one establishes all the  properties of the map $\al\to C_{\al,p}$, including continuity and the property \eqref{320}. Note that the same example as  in \eqref{d:18} provides the necessary ``almost'' maximizers, even though the power $p$ may not be an integer. This allows us to state the fnal result concerning $\phi_{\al,p}$ 
as follows.
\begin{proposition}
	\label{prop:65} 
	Let $p\in (2,6]$, $\al>0$ and the waves $\phi_{\al,p}$ are defined by a rescaling of $\vp_{\al,p}$ in \eqref{398}. Then, 
	\begin{itemize}
		\item $\al\to \|\phi_{\al,p}\|$ is a strictly increasing and continuing function, so that 
		$$
		\lim_{\al\to 0+} \|\phi_{\al,p}\|=0, \lim_{\al\to \infty} \|\phi_{\al,p}\|=\infty
		$$
		\item $\phi_{\al,p}$ is spectrally stable in the sense of Definition \ref{defi:10}. 
	\end{itemize}
\end{proposition}
   The justification for the stability of $\phi_{\al,p}$ is again by Corollary \ref{cor:10}. In fact, all the elements are already essentially established. Indeed, the linearized operator is 
   $$
   \cl_+=(D-1)^2+\om_{\al,p} - (p-1) |\phi_{\al,p}|^{p-1},
   $$
   which has the property $\cl_+[\phi']=0$, in addition to $\cl_+|_{\{\phi\}^\perp}\geq 0$, which is equivalent, through the scaling formula \eqref{398} of the corresponding property for $\cm_+$, established in Proposition \ref{prop:80}.

   \subsection{Generalized Benjamin waves as maximizers of a Sobolev embedding type inequality}
   \label{sec:4.2}
   In this section, we construct solutions of \eqref{b:50} for all powers $p: 2<p<\infty$, and not only for $L^2$ subcritical ones. On the flip side, and in sharp contrast with our contruction in Section \ref{sec:4.1}  it will turn out that some of these waves (generally for large enough powers of $p$) will be unstable. 
   
   More specifically, for $2<p\leq \infty$, 
   \begin{equation}
   	\label{500}
   	  \|u\|_{L^p}\lesssim  \|u\|_{H^1}\lesssim (\|(D-1) u\|^2+\om \|u\|^2)^{\f{1}{2}},
   \end{equation}
   where we have started with the Sobolev inequality, followed by the equivalence of norms $\|u\|_{H^1}\sim (\|(D-1) u\|^2+\om \|u\|^2)^{\f{1}{2}}$. 
   We intentionally did not specify the dependence of the constants in the above inequality, but note that they clearly depend on both $p,\om$. We formulate the relevant inequality in the following specific form  
    \begin{equation}
   	\label{510}
   	\|u\|_{L^p}^2\leq D_{p,\om}  (\|(D-1) u\|^2+\om \|u\|^2).
   \end{equation}
   where $D_{p,\om}$ is the exact constant in it. That is, 
   $$
   D_{p,\om}=\sup_{u\neq 0} \f{\|(D-1) u\|^2+\om \|u\|^2}{\|u\|_{L^p}^2}
   $$
   Note that $D_{p,\om}$, while clearly dependent upon $p$, is bounded as $p\to \infty$. One way to see that is to note that \eqref{500} is valid for $p=\infty$, whence \eqref{500} can be obtained via a Riesz-Thorin interpolation (between $p=\infty$ and the trivial case $p=2$, where $D_{2, \om}\leq \om^{-1}$). As a consequence 
   \begin{equation}
   	\label{520} 
   	 D_{p, \om}\leq  D_{\infty, \om}^{1-\f{2}{p}}D_{2, \om}^{\f{2}{p}} \leq \cd_\om.
   \end{equation}
   has a  bound uniform in $p\in (2, \infty)$. 
   We are now ready for the main existence result. 
   \begin{proposition}
   	\label{prop:1001} 
   	Let $p\in (2, \infty)$, $\om>0$. Then, there exists a maximizer $\vp_{p,\om}$ of the inequality \eqref{510}. That is, 
   	$$
   	D_{p,\om}= \f{\|(D-1) \vp\|^2+\om \|\vp\|^2}{\|\vp\|_{L^p}^2}
   	$$
   	Furthermore, $\vp$ is a weak $H^1(\rone)$ (and hence $H^3(\rone)$ as explained previously) solution of the Euler-Lagrange equation 
   	\begin{equation}
   		\label{550} 
   		(D-1)^2 \vp +\om \vp - \f{1}{D_{p,\om}\|\vp\|_p^{p-2}} |\vp|^{p-2} \vp=0.
   	\end{equation}
   while the linearized operator $\cm_+$ 
   $$
   \cm_+:=(D-1)^2 +\om  - \f{p-1}{D_{p,\om}\|\vp\|_p^{p-2}} |\vp|^{p-2} 
   $$
   is positive on a co-dimension one subspace, that is $ \cm_+|_{\{|\vp|^{p-2}\vp\}^\perp}\geq 0$. 
   \end{proposition}

   \begin{proof}
   	The existence part is very similar to the approach taken in the proof of Proposition \ref{prop:80}. We take a minimizing sequence $f_n$, with the property $\|(D-1) f_n\|^2+\om \|f_n\|^2=1$. With that,  
   	$$
   	\lim_n \|f_n\|_p^2=D_{p,\om}
   	$$
   	Then, setting as before $g_n:= |(D-1) f_n|^2+\om f_n^2$,  one proceeds to rule out vanishing and splitting, whence tightness is the only remaining option for $g_n$. In particular, $\{g_n\}$  becomes (after eventual translation and a subsequence) a convergent sequence yielding a strong limit, which in turn produces a strong limit $\vp: \lim_n \|f_n-\vp\|_{H^1}=0$. This function $\vp$ produces the solution claimed herein, and we omit the details as they are identical or very close to the one presented in the proof of Proposition \ref{prop:80}. 
   	
   	Next,  consider the function 
   	$$
   	g(\eps)= \|(D-1) (\vp+\eps h)\|^2+\om \|\vp+\eps h\|^2 - D_{p,\om}\|\vp+\eps h\|_p^p.
   	$$
   	Clearly, $g(0)=0$, and for $\vp$ to be a minimizer, it is necessary that $g'(0)=0$. This yields the Euler-Lagrange equation \eqref{550}. Taking $h\perp \vp|^{p-2}\vp$ and exploiting the minimization property yields $g''(0)\geq 0$, which in turn is equivalent to the property $ \cm_+|_{\{|\vp|^{p-2}\vp\}^\perp}\geq 0$. 
   \end{proof}
   Clearly, one produces solutions $\phi$ of \eqref{b:50} as before, by taking 
   $$
   \phi:= \f{1}{D_{p,\om}^{\f{1}{p-2}}\|\vp\|_p} \vp.
   $$
   One can translate the results of Proposition \ref{prop:1001} to $\phi$ as follows. 
   \begin{corollary}
   	\label{cor:40} 
   Let $2<p<\infty$, $\om>0$. Then, $\phi$ is a classical solution of  \eqref{b:50} and the linearized operator satisfies 
   $$
   \cl_+=(D-1)^2 +\om  - (p-1) |\vp|^{p-2}, \cl_+|_{\{|\phi|^{p-2}\phi\}^\perp}\geq 0.
   $$
 In addition, 
   \begin{eqnarray}
   	\label{700} 
   	\|\phi\|_p= \f{1}{D_{p,\om}^{\f{1}{p-2}}}.
   \end{eqnarray}
   \end{corollary}
Here, we should mention that $\cl_+$ has exactly one negative eigenvalue, i.e. $n(\cl_+)=1$. Indeed, the property $\cl_+|_{\{|\phi|^{p-2}\phi\}^\perp}\geq 0$ guarantees $n(\cl_+)\leq 1$, while by  direct inspection 
$$
\dpr{\cl_+\phi}{\phi}=
\|(D-1)\phi\|^2+\om \|\phi\|^2 - (p-1) \int |\phi|^p=(2-p)\int |\phi|^p<0,
$$
shows that $n(\cl_+)\geq 1$, hence $n(\cl_+)=1$. 
\subsection{On the instability of the generalized Benjamin waves for large $p$}
In this section, we combine several different criteria to bear on the proof of the instability. More concretely, we consider the eigenvalue problem \eqref{b:60}. We apply the instability index count for it (i.e. formula \eqref{e:20}), the fact that $n(\cl_+)=1$ (established in Corollary  \ref{cor:40}),   as well as Proposition \ref{prop:43} (where we still need to establish the existence of ground states $\Psi_0, \psi_0$). Specifically, the instability index theory yields us that since $\phi'\in Ker[\cl_+]$, we have that the solution of 
$$
\p_x \cl_+ \eta=\phi',
$$
which is $\eta=\cl_+^{-1} \phi$ is a member of $gKer(\cl)$. In view of the non-degeneracy assumption, that is $Ker[\cl_+]=span[\phi']$, we see that this is the only element of $gKer(\cl)$, provided $\dpr{\cl_+^{-1} \phi}{\phi}\neq 0$.  Indeed, potential further elements need to solve 
$$
\p_x\cl_+ \tilde{\eta}=\eta=\cl_+^{-1} \phi.
$$
Testing the Fredholmness of such equation (i.e. dot product with $\phi$) yields 
$$
\dpr{\cl_+^{-1} \phi}{\phi}=\dpr{\p_x\cl_+ \tilde{\eta}}{\phi}=
-\dpr{ \tilde{\eta}}{\cl_+\phi'}=0.
$$
Finally, we apply Proposition \ref{prop:43}, which we can do, provided we can establish that $\cl|_{\{\phi\}^\perp}>0$ fails. Indeed, it  asserts that in fact $\dpr{\cl_+^{-1} \phi}{\phi}>0$, which also rules out further elements in $gKer[\cl]$. So, it turns out that the matrix $\cd$ in the instability index formula consists of one element and in fact 
$$
\cd_{11}=\dpr{\cl_+\eta}{\eta}=\dpr{\cl_+^{-1} \phi}{\phi}>0
$$
This yields that the right-hand side of \eqref{e:20} is equal to one, hence implying the instability. This completes the proof of Theorem \ref{theo:30}, provided we are able to establish the remaining claims. 
We do so now. 

The first unfinished task is to establish   that $\cl_+$ achieves its ground states. While this is a classical fact for standard Schr\"odinger operators, it appears that such a property is missing in the literature, for non-local differential operators, like the one here, involving $D$. We have 
\begin{lemma}
	\label{le:90} 
Consider the non-local operators $\cl= (D-1)^2 + V$, with domain $D(\cl)=H^2(\rone)$, and $V\in C^\delta(\rone)$, for some  $\de>0$, $\lim_{|x|\to \infty}  V(x)=0$. Then, assuming that 
$$
\mu=\inf\{\dpr{\cl u}{u}: u\in H^1(\rone):   \|u\|=1\}<0.
$$
 then, $\cl$ achieves its ground states. That is, there exists an element $\psi_0\in D(\cl)$, so that  $\cl\psi_0=\mu \psi_0$.  
 
 More generally, for a fixed element $g\in L^2(\rone)$ and 
 $$
 \mu_g=\inf\{\dpr{\cl u}{u}: u\in H^1(\rone): u\perp g,   \|u\|=1 \}<0
 $$
 there exists $\psi_g\in D(\cl)$, so that $\cl\psi_g=\mu_g \psi_g+\al g$, for some $\al$.   
\end{lemma}
We provide a detailed  proof of Lemma \ref{le:90} in the Appendix. 

The remaining piece of the proof of the instability is the verification of the failure of the property $\cl_+|_{\{\phi\}^\perp}\geq 0$. This is indeed the case, under certain conditions on the parameters, as described in the next result. 
\begin{lemma}
	\label{le:89} 
	The property $\cl_+|_{\{\phi\}^\perp}\geq 0$ fails for large enough values of $p$. In fact, if $p,\om$ satisfy 
	$$
	\f{p}{4}+\f{4}{p} > \f{1}{\om} + \f{5}{2}, 
	$$
	then, $\cl_+|_{\{\phi\}^\perp}\geq 0$ fails. 
\end{lemma}
\begin{proof}
	
Here, we make a concrete calculation regarding the property $\cl_+|_{\{\phi\}^\perp}\geq 0$, which in conjuction with $n(\cl_+)=1$ is in fact necessary for stability. Thus, if we show that the property $\cl_+|_{\{\phi\}^\perp}\geq 0$ fails, we would have established instability for the respective model. 

To this end, we construct $\eta\perp \phi$, and we will show that under certain circumstances (i.e. large enough $p$), $\dpr{\cl_+\eta}{\eta}<0$, whence instability occurs as elucidated above. Specifically, let 
$$
\eta:= x\phi'+\f{1}{2}\phi
$$
   which by direct inspection satisfies $\eta\perp \phi$. Next, we calculate  $\cl_+\eta$.  Namely, using the commutation formula \eqref{b:60}, we compute 
   \begin{eqnarray*}
   	\cl_+(x\phi') &=&  D^2(x\phi') - 2 D (x\phi') +(\om+1)(x\phi') - (p-1) |\phi|^{p-2} \phi x\phi'=\\
   	&=& x\p_x( (D-1)^2\phi+\om \phi - |\phi|^{p-2}\phi)-2\phi''-2 D\phi=-2\phi''-2 D\phi. \\
   	\cl_+(\f{1}{2}\phi) &=& \f{p-2}{2} \phi'' +(p-2) D\phi - \f{(p-2)(\om+1)}{2} \phi
   \end{eqnarray*}
   Alltogether, 
   \begin{equation}
   	\label{710} 
   	\cl_+\eta = \f{p-6}{2} \phi'' +(p-4) D\phi - \f{(p-2)(\om+1)}{2} \phi. 
   \end{equation}
Thus, as $\phi\perp \eta$, 
\begin{eqnarray*}
	\dpr{\cl_+\eta}{\eta}  &=&  
	\f{p-6}{2}\dpr{\phi''}{x\phi'+\f{1}{2}\phi}+(p-4) \dpr{D\phi}{x\phi'+\f{1}{2}\phi}=\f{6-p}{2} \|\phi'\|^2+\f{p-4}{2} \|D^{\f{1}{2}} \phi\|^2.
\end{eqnarray*}
Using the Pohozaev's identities, \eqref{b:70}and \eqref{b:80}, we can rewrite the last expression as 
\begin{equation}
	\label{730} 
		\dpr{\cl_+\eta}{\eta}  = (\om+1) \|\phi\|^2 - \left(\f{p}{4}+\f{4}{p}-\f{3}{2}\right) \|\phi\|_p^p.
\end{equation}
We claim that the expression in \eqref{730} will necessarily becomes negative for large enough $p$, thus proving our claim. Indeed, from \eqref{b:70} and \eqref{700}, we have that 
$$
\om \|\phi\|^2\leq \|(D-1)\phi\|^2+\om \|\phi\|^2 = \|\phi\|_p^p=\f{1}{D_{p,\om}^{\f{p}{p-2}}}.
$$
It follows that 
$$
	\dpr{\cl_+\eta}{\eta}  \leq \f{1}{D_{p,\om}^{\f{p}{p-2}}}\left(
	 \f{\om+1}{\om} -\left(\f{p}{4}+\f{4}{p}-\f{3}{2}\right) \right)<0,
$$
provided $\left(\f{p}{4}+\f{4}{p}-\f{3}{2}\right) > \f{\om+1}{\om}$.

\end{proof}


   \appendix
   
   \section{Proof of Lemma \ref{le:93}}
   The proof below of Lemma \ref{le:93} is lifted verbatim from \cite{PS18}. 	
   First, we can without loss of generality assume that $\|\xi_0\|=1$.  
   	Consider 
   	$
   	\eta:=\ch^{-1} \xi_0  - \dpr{\ch^{-1} \xi_0}{\xi_0} \xi_0\perp \xi_0. 
   	$
   	It follows that 
   	\begin{eqnarray*}
   		0 &\leq & \dpr{\ch\eta}{\eta}=\dpr{\ch[ \ch^{-1} \xi_0  - \dpr{\ch^{-1} \xi_0}{\xi_0} \xi_0]}{ \ch^{-1} \xi_0  - \dpr{\ch^{-1} \xi_0}{\xi_0} \xi_0}= \\
   		&=& \dpr{\xi_0- \dpr{\ch^{-1} \xi_0}{\xi_0}  \ch \xi_0}{ \ch^{-1} \xi_0  - \dpr{\ch^{-1} \xi_0}{\xi_0} \xi_0}=\\
   		&=& - \dpr{\ch^{-1} \xi_0}{\xi_0} +\dpr{\ch^{-1} \xi_0}{\xi_0}^2 \dpr{\ch\xi_0}{\xi_0} \leq - \dpr{\ch^{-1} \xi_0}{\xi_0},
   	\end{eqnarray*}
   	where we have used the assumption $\dpr{\ch\xi_0}{\xi_0} \leq 0$. 
   	It follows that 
   	$\dpr{\ch^{-1} \xi_0}{\xi_0}\leq 0$, which is the claim. 
    
   \section{Proof of Lemma \ref{le:90}} 
   We choose a minimizg sequence for $\mu$, say $f_n\subset H^2(\rone): \|f_n\|=1$, so that $\lim_n \dpr{\cl f_n}{f_n}=\mu<0$. 
Since $|\int V|f_n|^2|\leq \max_{\rone} |V|$, we have a bound on $\|f_n\|_{H^1}$, since 
$$
\dpr{(\cl-\mu)f_n}{f_n}\geq \|(D-1) f_n\|^2-\mu \|f_n\|^2 - \|V\|_{L^\infty}.
$$
It follows that $\limsup_n \|f_n\|_{H^1}^2\leq C_\mu\limsup_n ( \|(D-1) f_n\|^2-\mu \|f_n\|^2)\leq C \|V\|_{L^\infty}$. So, we have a bounded sequence in $H^1(\rone)$, so we may take a weakly convergent (in $H^1$) subsequence $f_n\rightharpoonup f$. By the properties of $V$,  
$$
\lim_{R\to \infty} \sup_n  \int_{|x|>R} V^2(x) f_n^2(x) dx=0. 
$$
Also, an elementary estimate shows that since $V\in C^\de(\rone)$, we have
\begin{eqnarray*}
	\|V(\cdot+h) f_n(\cdot+h)-V f_n(\cdot)\| &\leq &  \sup_x |V(x+h)-V(x)|\|f_n\|+\|V\|_{L^\infty} \|f_n(\cdot+h)-f_n(\cdot)\|\leq \\
	&\leq & C |h|^\de \|f_n\|  + \|V\|_{L^\infty} |h| \|f_n\|_{H^1}.
\end{eqnarray*}
Taking $\lim_{h\to 0}\sup_n $, 
$$
\lim_{h\to 0} \sup_n \|V(\cdot+h) f_n(\cdot+h)-V f_n(\cdot)\| =0.
$$
In conclusion, $\{V f_n\}_n$ is a pre-compact sequence in $L^2(\rone)$. One may then select a strongly convergent (in $L^2$)  subsequence of  $Vf_n$. Without loss of generality, we may assume that $Vf_n$ itself is strongly convergent and it converges to its weak limit $Vf$, i.e. $\lim_n \|V f_n-V f\|=0$. As a consequence 
$$
	\lim_n \int V(x) f_n^2 =	\lim_n [ \int f_n(V f_n-V f) + \int f(V f_n - V f)]+ \int V  f^2 = 
	 \int V  f^2 
$$
Up to this point, we have not ruled out the situation $f=0$, and we do so now. Indeed, assuming that $f=0$, we have 
$$
\mu=\lim_n \dpr{\cl f_n}{f_n} \geq \liminf_n \int |(D-1) f_n|^2+\lim_n \int V(x) f_n^2(x) dx\geq 0,
$$
a contradiction with the assumption $\mu<0$. Hence $f\neq 0$, but note that by the the lower semi-continuity of the $L^2$ norm with respect to the weak convergence in $L^2$, $\|f\|\leq \liminf_n \|f_n\|=1$. 

In addition, since  $f_n\rightharpoonup f$ in $H^1$, it follows that $(D-1) f_n \rightharpoonup (D-1)f$ weakly in $L^2$, whence by the lower semi-continuity of the $L^2$ norm, 
\begin{eqnarray*}
	\liminf_n \int |(D-1) f_n|^2 \geq \int |(D-1) f|^2
\end{eqnarray*}
It follows that 
$$
\mu=\liminf_n \dpr{\cl f_n}{f_n}\geq \dpr{\cl f}{f}.
$$
We now show that either of the strict inequalities $\|f\|<1$ or $\mu>\dpr{\cl f}{f}$ leads to a contradiction (and hence both are equalities). Indeed, assuming that $\dpr{\cl f}{f}<\mu$ and letting $\al:=\f{1}{\|f\|}\geq 1$, we obtain (recall $\mu<0$), 
$$
\dpr{\cl (\al  f)}{\al f}=\al^2 \dpr{\cl f}{f}<\al^2 \mu\leq \mu.
$$
It follows that $f_\al:=\al f, \|f_\al\|=1$ has the property $\dpr{\cl f_\al}{f_\al}<\mu$, a contradiction with the definition of $\mu$. Similar contradiction is reached, if $\dpr{\cl f}{f}\leq \mu$, but $\al>1$. So, $\|f\|=1$ and $\dpr{\cl f}{f}=\mu$. It follows that an $L^2$ strong convergence $f_n\to f$ is true (at least on a subsequence) and also $f$ is a solution to the constrained minimization problem 
$$
\mu=\inf\{\dpr{\cl u}{u}: \|u\|=1\}.
$$
An easy Euler-Lagrange argument yields  that $\cl f=\mu f$. 
More generally, same proof works in the case of $\mu_g$, except in the last step, the Euler-Lagrange equation for 
$$
\mu_g=\inf\{\dpr{\cl u}{u}: u\perp g, \|u\|=1\}<0
$$
is  that $\dpr{\cl f-\mu f}{h}=0$, for all test functions  with the property $h: h\perp g$. It follows that there exists $\al\in\rone$, so that $\cl f=\mu f+\al g$, as claimed.


\begin{thebibliography}{99}
	
\bibitem{ADM}	M. Abdallah, M. Darwich, L. Molinet, \emph{On the Uniqueness and Orbital Stability of Slow and Fast Solitary Wave Solutions of the Benjamin Equation}, preprint, avaliable at https://arxiv.org/pdf/2404.04711
	
 
\bibitem{ABR} J. Albert,  J. Bona, J.M. Restrepo, 
\emph{Solitary-wave solutions of the Benjamin equation}, {\em SIAM J. Appl. Math.}, {\bf 59}, (1999), no.6, p. 2139--2161.


\bibitem{Ben} T.B. Benjamin, Lectures on nonlinear wave motion, Nonlinear Wave Motion, A. C. Newell, ed., AMS, Providence, R. I., 15 (1974), p. 3--47.

 \bibitem{Ben1} 
 T. B. Benjamin,  \emph{A new kind of solitary waves}, {\em J. Fluid Mechanics}, {\bf 245},  (1992),  p. 401--411.
 
 \bibitem{Ben2}  T. B. Benjamin, \emph{Solitary and periodic waves of a new kind}, {\em Phil. Trans. Roy. Soc. London Ser. A} {\bf 354}, 
  (1996), p. 1775--1806.
  
  \bibitem{BL} J. L. Bona, Y. A. Li, \emph{Decay and analyticity of solitary waves}, {\em J. Math. Pures Appl.} {\bf  76}, (1997),  p. 377--430.
  
\bibitem{BC}  H.  Chen,  J. Bona,\emph{
  Existence and asymptotic properties of solitary-wave solutions of Benjamin-type equations}, {\em Adv. Differential Equations}, {\bf 3}, (1998), no.1, p.  51--84.
  

    \bibitem{KKS1} T. M. Kapitula, P. G. Kevrekidis, B. Sandstede, \emph{Counting eigenvalues
    	via Krein signature in infinite-dimensional Hamitonial systems,}  {\em Physica D}, {\bf 3-4}, (2004), p. 263--282.

    \bibitem{KKS2} T. Kapitula, P. G. Kevrekidis, B. Sandstede, \emph{ Addendum: "Counting eigenvalues via the Krein signature in
    	infinite-dimensional Hamiltonian systems'' [Phys. D 195 (2004), no. 3-4,
    	263--282] }
    {\em Phys. D}  {\bf  201}  (2005),  no. 1-2, 199--201.


\bibitem{Kap} T. Kapitula,  K. Promislow,
Spectral and dynamical stability of nonlinear waves.  {\em Applied Mathematical Sciences}, {\bf  185},  Springer, New York, 2013.

 
\bibitem{LZ} Z. Lin, C.  Zeng, {\emph Instability, index theorem, and exponential trichotomy for Linear Hamiltonian PDEs},  {\em Mem. Amer. Math. Soc.}, {\bf 275}, (2022), no.1347.


\bibitem{Lions} P.L. Lions, {\emph The concentration-compactness principle in the calculus of variations. The locally compact case, part I}, 
{\em  Annales de l'I. H. P.},  section C, tome 1, no 2 (1984), p. 109-145. 

 \bibitem{Pava1} J. A. Pava, 
\emph{Existence and stability of solitary wave solutions of the Benjamin equation.} {\em J. Differential Equations}, {\bf 152}, (1999), no.1, p. 136--159.

  \bibitem{Pel} D. Pelinovsky, \emph{Inertia law for spectral stability of solitary waves in coupled nonlinear Schr\"odinger equations.} {\em Proc. R. Soc. Lond. Ser. A Math. Phys. Eng. Sci.} {\bf  461} (2005) , no. 2055, p. 783--812.


 
\bibitem{PS18} I. Posukhovskyi,  A. Stefanov, \emph{
On the normalized ground states for the Kawahara equation and a fourth order NLS}, {\em Discrete Contin. Dyn. Syst.} {\bf 40}, 
(2020), no.7, p. 4131--4162.




\end{thebibliography}
\end{document}